\documentclass[11pt,reqno]{amsart} 

\usepackage{graphicx}

\usepackage[top=1in,bottom=1in,left=1in,right=1in]{geometry}

\usepackage{xcolor}

\usepackage{enumitem}

\usepackage[labelfont=bf,font=sf]{caption}

\usepackage[font={sf,small},labelfont=md]{subfig}

\usepackage[noadjust,nosort]{cite}

\usepackage{amsmath, amsthm, amssymb}
\usepackage{enumerate}
\usepackage{tikz}
\usepackage{mathtools}

\usepackage{polynom}

\usepackage{multicol}

\usepackage{dsfont}
	\newcommand{\one}{\mathds{1}}

\usepackage{framed}

\usepackage{hyperref}
	\hypersetup{colorlinks=true,linkcolor=blue,citecolor=blue,urlcolor=black}
	
\allowdisplaybreaks

\numberwithin{equation}{section}


\newcommand{\eq}[1]{\begin{align*} #1 \end{align*}}
\newcommand{\eeq}[1]{\begin{align} \begin{split} #1 \end{split} \end{align}}

\newcommand{\N}{\mathbb{N}}

\newcommand{\R}{\mathbb{R}}

\newcommand{\Z}{\mathbb{Z}}

\newcommand{\EE}{\mathbf{E}}

\newcommand{\PP}{\mathbf{P}}

\renewcommand{\a}{\mathcal{A}}

\newcommand{\f}{\mathcal{F}}
\newcommand{\g}{\mathcal{G}}

\renewcommand{\k}{\mathcal{K}}
\renewcommand{\l}{\mathcal{L}}
\newcommand{\m}{\mathcal{M}}
\newcommand{\n}{\mathcal{N}}

\newcommand{\p}{\mathcal{P}}

\renewcommand{\r}{\mathcal{R}}
\newcommand{\s}{\mathcal{S}}
\renewcommand{\t}{\mathcal{T}}
\renewcommand{\u}{\mathcal{U}}

\newcommand{\w}{\mathcal{W}}

\newtheorem{thm}{Theorem}[section]
\newtheorem{prop}[thm]{Proposition}

\newtheorem{lemma}[thm]{Lemma}
\newtheorem{theirthm}{Theorem}[section]

\theoremstyle{definition}

\newtheorem{remark}[thm]{Remark}

\def\eps{\varepsilon}
\def\vphi{\varphi}

\newcommand{\vc}[1]{{\boldsymbol #1}}

\newcommand{\wt}[1]{\widetilde{#1}}

\newcommand{\givenk}[3][]{#1[ #2 \: #1| \: #3 #1]} 
\newcommand{\givenp}[3][]{#1( #2 \: #1| \: #3 #1)} 

\newcommand{\Givenp}[3][]{#1( #2 \: ; \: #3 #1)} 

\newcommand{\dd}{\mathrm{d}}

\DeclareMathOperator{\diam}{diam}

\renewcommand{\thefootnote}{\fnsymbol{footnote}}

\title{Localization of directed polymers with general reference walk}


\author{Erik Bates}
\thanks{This research was partially supported by NSF grant DGE-114747} 
\address{\newline Department of Mathematics \newline Stanford University \newline 450 Serra Mall, Bldg 380 \newline Stanford, CA 94305 \newline \textup{\tt ewbates@stanford.edu}}

\begin{document}
\bibliographystyle{acm}

\renewcommand{\thefootnote}{\arabic{footnote}} \setcounter{footnote}{0}

\begin{abstract}
Directed polymers in random environment have usually been constructed with a simple random walk on the integer lattice.
It has been observed before that several standard results for this model continue to hold for a more general reference walk.
Some finer results are known for the so-called long-range directed polymer in which the reference walk lies in the domain of attraction of an $\alpha$-stable process.
In this note, low-temperature localization properties recently proved for the classical case are shown to be true with any reference walk.
First, it is proved that the polymer's endpoint distribution is asymptotically purely atomic, thus strengthening the best known result for long-range directed polymers.
A second result proving geometric localization along a positive density subsequence is new to the general case.
The proofs use a generalization of the approach introduced by the author with S. Chatterjee in a recent manuscript on the quenched endpoint distribution; this generalization allows one to weaken assumptions on the both the walk and the environment.
The methods of this paper also give rise to a variational formula for free energy which is analogous to the one obtained in the simple random walk case.
\end{abstract}

\maketitle

\setcounter{footnote}{0} 
\section{Introduction}
The probabilistic model of directed polymers in random environment was introduced by Imbrie and Spencer \cite{imbrie-spencer88} as a reformulation of Huse and Henley's approach \cite{huse-henley85} to studying the phase boundary of the Ising model in the presence of random impurities.
In its classical form, the model considers a simple random walk (SRW) on the integer lattice $\Z^d$, whose paths---considered the ``polymer"---are reweighted according to a random environment that refreshes at each time step.
Large values in the environment tend to attract the random walker and possibly force localization phenomena; this attraction grows more effective in lower dimensions and at lower temperatures.
On the other hand, the random walk's natural dynamics favor diffusivity.
Which of these competing features dominates asymptotically is a central question in the study of directed polymers.  
Much progress has been made over the last thirty years in understanding polymer behavior;
for a comprehensive and up-to-date survey, the reader is referred to the recent book by Comets \cite{comets17}.

In \cite{comets07}, Comets initiated the study of long-range directed polymers.
In this model, the simple random walk is replaced by a general random walk capable of superdiffusive motion.
More specifically, it is assumed that the walk belongs to the domain of attraction of an $\alpha$-stable law for some $\alpha \in (0,2]$.
For example, any walk having increments with a finite second moment belongs to the $\alpha=2$ case.
Under this assumption the long-range polymer can model the behavior of 
heavy-tailed walks, such as L\'evy flights, when placed in an inhomogeneous random environment.
Indeed, L\'evy flights in random potentials have been used to study chemical reactions \cite{chen-deem01} and particle dispersions \cite{sokolov-mai-blumen97,brockmann-geisel03}.
Moreover, their continuous-time analogs, L\'evy processes, appear in a variety of disciplines including fluid mechanics, solid state physics, polymer chemistry, and mathematical finance \cite{nielsen-mikosch-resnick01}.
This is relevant because $\alpha$-stable polymers are known to obey a scaling CLT at sufficiently high temperatures (\cite[Theorem 4.2]{comets07} and \cite[Theorem 1.9]{wei16}), which generalizes the Brownian CLT proved in \cite[Theorem 1.2]{comets-yoshida06} when the reference walk is SRW.

Interestingly, universal behaviors have also appeared at low temperatures, where the system exists in  a ``disordered" phase.
Part of the work in \cite{comets07,wei16} was to extend localization results known for polymers constructed from a SRW to those constructed with $\alpha$-stable reference walks.
This paper continues the advance in this direction by proving that in the localization regime, certain qualitative behaviors of the polymer's endpoint distribution are the same for \textit{any} reference walk.
Namely, the strongest forms of localization known for arbitrary environment and arbitrary dimension, which were only recently proved for the SRW case in \cite{bates-chatterjee17}, are established here for the general case.

The organization of the remaining introduction is as follows.
After the polymer model is formally introduced, we will recall the relevant facts from the literature in Section \ref{overview_known_results} and state our main results in Section \ref{results_this_paper}.
The proof strategy is outlined in Section \ref{outline_methods_intro}, which describes how the approach used in \cite{bates-chatterjee17} must be expanded to work for general polymers.
Finally, Section \ref{related_aspects} offers references on other fronts of progress in both the short- and the long-range settings.

\subsection{The model}
Let $d$ be a positive integer, to be called the \textit{spatial} or \textit{transverse dimension}. 
Let $P$ denote the law of the \textit{reference walk}, which is a homogeneous random walk $(\omega_i)_{i\geq0}$ on $\Z^d$.
To be precise, we state the assumptions on $P$:
\eeq{
P(\omega_0 = 0) = 1, \qquad P\givenp{\omega_{i+1} = x}{\omega_i = y} = P(\omega_1 = x - y) \eqqcolon P(y,x) < 1. \label{walk_assumption}
}
Next we introduce a collection of i.i.d.~random variables $\vc \eta = (\eta(i,x) : i \geq 1, x \in \Z^d)$, called the \textit{random environment}, supported on some probability space $(\Omega,\f,\PP)$.
We will write $E$ and $\EE$ for expectation with respect to $P$ and $\PP$, respectively.
Finally, let $\beta > 0$ be a parameter representing \textit{inverse temperature}.
Then for $n\geq0$, the \textit{quenched polymer measure} of length $n $ is the Gibbs measure $\rho_n$ defined by
\eq{
\rho_n(\dd \omega) = \frac{1}{Z_n} e^{-\beta H_n(\omega)}\ P(\dd \omega), \quad \text{where} \quad H_n(\omega) \coloneqq -\sum_{i = 1}^n \eta(i,\omega_i).
}
The normalizing constant
\eq{
Z_n \coloneqq E(e^{-\beta H_n(\omega)}) = \sum_{x_1,\dots,\, x_n \in \Z^d} \exp\bigg(\beta \sum_{i = 1}^n \eta(i,x_i)\bigg)\prod_{i=1}^n P(x_{i-1},x_i), \quad x_0 \coloneqq 0,
}
is called the \textit{quenched partition function}.
A fundamental quantity of the system is calculated from this constant, namely the \textit{quenched free energy},
\eq{
F_n \coloneqq \frac{\log Z_n}{n}.
}
We specify ``quenched" to indicate that the randomness from the environment has not been averaged out.
That is, $\rho_n$, $Z_n$, and $F_n$ are each random processes with respect to the filtration
\eq{
\f_n \coloneqq \sigma(\eta(i,x) : 1 \leq i \leq n, x \in \Z^d), \quad n \geq 0.
}
When $Z_n$ is replaced by its expectation $\EE(Z_n)$, one obtains the \textit{annealed} free energy,
\eq{
\frac{\log \EE(Z_n)}{n} = \frac{\log (\EE\, e^{\beta \eta})^n}{n} = \log \EE(e^{\beta \eta}) \eqqcolon \lambda(\beta),
}
where $\eta$ denotes (here and henceforth) a generic copy of $\eta(1,0)$.
Notice that $\lambda(\cdot)$ depends only on the law of the environment, which we denote $\mathfrak{L}_\eta$.
We will assume finite ($1+\eps$)-exponential moments,
\eeq{ \label{mgf_assumption}
0 < \beta < \beta_{\max} \coloneqq \sup\{t \geq 0: \lambda(\pm t) < \infty\},
}
so that $Z_n$ has finite $\pm(1+\eps)$-moments at the given inverse temperature.
Otherwise $\mathfrak{L}_\eta$ is completely general, although to avoid trivialities, we will always assume that $\eta$ is not an almost sure constant, and that $\beta$ is strictly positive.

\subsection{Overview of known results} \label{overview_known_results}



Two fundamental facts are the convergence of the free energy and the existence of a corresponding phase transition.
The first result below was initially shown by Carmona and Hu \cite[Proposition 1.4]{carmona-hu02} for a Gaussian environment and by Comets, Shiga, and Yoshida \cite[Proposition 2.5]{comets-shiga-yoshida03} for a general environment, and later proved by Vargas \cite[Theorem 3.1]{vargas07} under weaker assumptions.

\begin{theirthm}[{\cite[Proposition 2.5]{comets-shiga-yoshida03}}]\label{convergence_background}
Let $P$ be SRW and assume $\lambda(t) < \infty$ for all $t \in \R$.
Then there exists a deterministic constant $p(\beta)$ such that
\eq{
\lim_{n \to \infty} F_n = p(\beta) \quad \mathrm{a.s.}
}
\end{theirthm}

\begin{theirthm}[{\cite[Theorem 3.2(b)]{comets-yoshida06}}] \label{phase_transition_background}
Let $P$ be SRW and assume $\lambda(t) < \infty$ for all $t \in \R$.
Then there exists a critical value $\beta_c \in [0,\infty]$ such that
\eq{
0 \leq \beta \leq \beta_c \quad &\Rightarrow \quad p(\beta) = \lambda(\beta) \\
\beta > \beta_c \quad &\Rightarrow \quad p(\beta) < \lambda(\beta).
}
\end{theirthm}

The high temperature region $(0,\beta_c]$ is often called the ``delocalized" phase, while the low temperature interval $(\beta_c,\infty)$ forms the ``localized" phase.
The names are justified by the next result. 

\begin{theirthm}[{\cite[Corollary 2.2]{comets-shiga-yoshida03}}] \label{localization_background}
Let $P$ be SRW and assume $\lambda(t) < \infty$ for all $t \in \R$.
Define the (random) set
\eq{
\a_i^\eps \coloneqq \{x \in \Z^d : \rho_i(\omega_i = x) > \eps\}.
}
Then $p(\beta) < \lambda(\beta)$ if and only if there exists $\eps > 0$ such that
\eeq{ \label{eps_atom_eps}
\liminf_{n\to\infty} \frac{1}{n}\sum_{i=0}^{n-1} \one_{\{\a_i^\eps \neq \varnothing\}} \geq \eps \quad \mathrm{a.s.}
}
\end{theirthm}

In \cite{vargas07}, $\a_i^\eps$ is called the set of $\eps$-atoms.
In other words, a polymer is localized when it contains macroscopic atoms that persist with positive density.
This result was generalized to the case of $0 < \beta < B\coloneqq \sup\{t \geq 0 : \lambda(t) < \infty\}$ by Vargas \cite[Theorem 3.6]{vargas07}, whose argument was extended to general $P$ by Wei \cite[Theorem 1.17]{wei16}.
Vargas also proved a stronger localization \eqref{eps_atom_delta} for parameters sufficiently close to $B$.

\begin{theirthm}[{\cite[Theorem 3.7]{vargas07}}] \label{partial_apa_background}
Let $P$ be SRW.
Assume $B > 0$ and $\lim_{t \to B} \lambda(t)/t = \infty$.
Then for every $\delta < 1$, there exists $\eps > 0$ and $\beta_0 \in (0,B)$ such that
\eeq{ \label{eps_atom_delta}
\beta \in (\beta_0,B) \quad \Rightarrow \quad \liminf_{n\to\infty}\frac{1}{n}\sum_{i = 0}^{n-1} \rho_i(\omega_i \in \a_i^\eps)  \geq \delta \quad \mathrm{a.s.}
}
\end{theirthm}

We conclude this section by recalling a simple sufficient condition for localization.
Notice that the condition depends only on the entropy of the random walk increment, not its actual distribution.
One can therefore interpret the result as providing a temperature threshold below which concentration of the polymer measure is strong enough to overcome any superdiffusivity of the random walk.

\begin{theirthm}[{\cite[Proposition 5.1]{comets07}}] \label{localization_sufficient_background}
Assume $\lambda(\beta) < \infty$, and also that $q(x) \coloneqq P(\omega_1 = x)$ has finite entropy.
If
\eq{
\beta \lambda'(\beta) - \lambda(\beta) > -\sum_{x \in \Z^d} q(x)\log q(x),
}
then $p(\beta) < \lambda(\beta)$.
\end{theirthm}


\subsection{Results of this paper} \label{results_this_paper}
The temperature range at which localization occurs depends on $P$, but to the extent described in our two main results below, the type of localization does not.
We will also prove Theorems \ref{convergence_background} and \ref{phase_transition_background} in the general setting so that the statements of localization results make sense.

First, it was recently shown in \cite{bates-chatterjee17} that $\beta_0$ in \eqref{eps_atom_delta} can be taken equal to $\beta_c$.
Alternatively, the right-hand side of \eqref{eps_atom_eps} can be taken arbitrarily close to $1$.
These realizations continue to hold true when $P$ is any walk satisfying \eqref{walk_assumption}, as demonstrated in the following improvement of Theorems \ref{localization_background} and \ref{partial_apa_background}, thus extending \cite[Theorem 1.1]{bates-chatterjee17} beyond the SRW case.

\begin{thm} \label{apa_new}
Assume \eqref{walk_assumption} and \eqref{mgf_assumption}.
\begin{itemize}
\item[(a)] If $p(\beta) < \lambda(\beta)$, then for every sequence $\eps_i \to 0$,
\eq{
\lim_{n\to\infty}\frac{1}{n}\sum_{i=0}^{n-1} \rho_i(\omega_i \in \a_i^{\eps_i}) = 1 \quad \mathrm{a.s.}
}
\item[(b)] If $p(\beta) = \lambda(\beta)$, then there exists a sequence $\eps_i \to 0$ such that
\eq{
\lim_{n\to\infty}\frac{1}{n}\sum_{i=0}^{n-1} \rho_i(\omega_i \in \a_i^{\eps_i}) = 0 \quad \mathrm{a.s.}
}
\end{itemize}
\end{thm}


This result is later stated and proved as Theorem \ref{total_mass}.
In case (a) above, we say that the sequence $(\rho_i(\omega_i \in \cdot))_{i\geq 0}$ is \textit{asymptotically purely atomic}\footnote{To the author's knowledge, this terminology was first used by Vargas, who proved in \cite[Corollary 3.3]{vargas07} that if $\lambda(\beta) = \infty$, then condition  (a) of Theorem \hyperref[apa_new]{1.1} holds with $P = $ SRW and ``in probability" replacing ``almost surely".}, which is meant to indicate that \textit{all} of the limiting mass is atomic, not just an $\eps$-fraction.
Through similar methods it can be shown that low-temperature polymer measures exhibit a type of geometric localization in addition to atomic localization.
That is, the macroscopic atoms can be seen close together.
Our second main result generalizes \cite[Theorem 1.2]{bates-chatterjee17}.

\begin{thm} \label{geometric_new}
Assume \eqref{walk_assumption} and \eqref{mgf_assumption}.
Let $\g_{\delta,K}$ denote the set of probability mass functions on $\Z^d$ that assign measure greater than $1-\delta$ to some subset of $\Z^d$ having $\ell^1$-diameter at most $K$, and define $f_i(x) \coloneqq \rho_i(\omega_i = x)$.
\begin{itemize}
\item[(a)] If $p(\beta) < \lambda(\beta)$, then for any $\delta > 0$, there exists $K < \infty$ and $\theta > 0$ such that
\eq{
\liminf_{n\to\infty} \frac{1}{n}\sum_{i=0}^{n-1} \one_{\{f_i \in \g_{\delta,K}\}} \geq \theta \quad \mathrm{a.s.}
}
\item[(b)] If $p(\beta) = \lambda(\beta)$, then for any $\delta \in (0,1)$ and any $K > 0$,
\eq{
\lim_{n\to\infty}\frac{1}{n}\sum_{i=0}^{n-1} \one_{\{f_i \in \g_{\delta,K}\}} = 0 \quad \mathrm{a.s.}
}
\end{itemize}
\end{thm}

This result appears as Theorem \ref{localized_subsequence} in the sequel.
In case (a), we say that the sequence of measures $(\rho_i(\omega_i \in \cdot))_{i\geq 0}$ exhibits \textit{geometric localization with positive density}.
The constant $\theta$ gives a lower bound on the density of endpoint distributions displaying the desired level of concentration.
If $\theta$ can be taken equal to $1$, then essentially all large atoms remain close, so we would say the sequence is \textit{geometrically localized with full density}.
That this is the case is sometimes called the ``favorite region conjecture", although one possibility is that it is true only in low dimensions.
The only model for which it is known to be true is the one-dimensional log-gamma polymer \cite{seppalainen12}, whose exact solvability was leveraged in \cite{comets-nguyen16} to obtain a limiting law for the endpoint distribution.
The general method of \cite{bates-chatterjee17} also was to establish a limiting law for the endpoint distribution, but in a more abstract compactification space.
A sufficient condition for the existence of a favorite region was given \cite[Theorem 7.3(b)]{bates-chatterjee17}, although a way to check the condition is not currently known.

\subsubsection{Comment on moment assumptions}
While Theorems \ref{apa_new} and \ref{geometric_new} were shown in \cite{bates-chatterjee17} in the case $P$ is SRW, that paper also critically assumed $\lambda(\pm 2\beta) < \infty$.
Here we are establishing the results not only for an arbitrary reference walk, but also under the weaker hypothesis \eqref{mgf_assumption}.
In this way, the present study both expands previous results to the setting of a general walk, and optimizes assumptions even in the classical case.
While this latter point may seem minor, it actually permits parameter ranges for which the regions of interest had been previously left out.
For example, if the random environment consists of exponential random variables, say with rate parameter $1$, then
\eq{
\lambda(\beta) = -\log(1-\beta), \quad \beta < \beta_{\max} = 1.
}
The assumption $\lambda(\pm2\beta) < \infty$ only covers $\beta < \frac{1}{2}$.
By Theorem \ref{localization_sufficient_background}, we know that there is localization (in the SRW case) at any $\beta$ such that
$\beta\lambda'(\beta) - \lambda(\beta) 
 > \log(2d)$.
But for any positive integer $d$ (in particular $d \geq 3$, where localization does not necessarily occur),
\eq{
\lambda'(1/2)/2 - \lambda(1/2) = 1 + \log(1/2) < \log(2d).
}
Therefore, the temperature regime for which localization is actually known has no intersection with the regime in which hypothesis $\lambda(\pm2\beta) < \infty$ is true.
This paper eliminates this discrepancy by assuming only \eqref{mgf_assumption}.
The technical challenges incurred are non-trivial, but the fact that they can be overcome reflects the generality with which our methodology may be useful, possibly in contexts other than polymer models.
Indeed, this feature is itself a motivation for the present study.

\subsection{Outline of methods} \label{outline_methods_intro}
We write $f_n$ to denote the probability mass function on $\Z^d$ of the (random) endpoint distribution $\rho_n(\omega = \cdot)$ for the length-$n$ polymer.
It is not difficult to see that $(f_n)_{n\geq0}$ is an $\ell^1(\Z^d)$-valued Markov process with respect to the canonical filtration $(\f_n)_{n\geq0}$ (see Section \ref{update_map}).
Unfortunately, the space $\ell^1(\Z^d)$ is too large to establish any type of convergence.
More to the point, we cannot expect any tightness for $(\rho_n)_{n\geq0}$.

In \cite{bates-chatterjee17}, this issue is resolved by constructing a compact space in which to embed the Markov chain.
Specifically, the mass functions on $\Z^d$ are identified with mass functions on $\N \times \Z^d$, which themselves are part of a larger space:
\eq{
\s_0 \coloneqq \Big\{f : \N \times \Z^d \to [0,1] : \sum_{u \in \N \times \Z^d} f(u) \leq 1\Big\}.
}
A pseudometric $d_*$ can be defined on $\s_0$ such that the resulting quotient space $\s \coloneqq \s_0/(d_*=0)$ is compact.
After not too much work, one can show that the equivalence classes under $d_*$ are the orbits under translations.
In other words, the only information that is lost by passing from $\ell^1(\Z^d)$ to $\s$ is the location of the origin.
We call the elements of $\s$ \textit{partitioned subprobability measures}.

In hindsight, then, the only obstructions to compactness were translations, which form a \textit{non}-compact group under which $\Z^d$ is invariant.
Observed also for more general unbounded domains, this phenomenon is often called ``concentration compactness" in the study of PDEs and calculus of variations.
In the 1980s, Lions \cite{lions84I,lions84II,lions85I,lions85II} made highly effective use of the concept by transferring problems to compact spaces using concentration functions introduced by L\'evy \cite{levy54}, although the idea for this type of compactification scheme goes back to work of Parthasarathy, Ranga Rao and Varadhan~\cite{parthasarathy-rangarao-varadhan62}.
The particular construction in \cite{bates-chatterjee17} was inspired in part by a continuum version used by Mukherjee and Varadhan \cite{mukherjee-varadhan16} to prove large deviation principles for Brownian occupation measures.

Once the Markov chain of endpoint distributions is embedded in a compact space, a few key ideas can be wielded to great effect:
\begin{itemize}
\item[(i)] Consider the ``update map" $\t$ which receives a starting position for the chain and outputs the law after one step.
Provided $\t$ is continuous, the chain's empirical measure will converge to stationarity.
\item[(ii)] A convexity argument shows that the starting configuration of $P(\omega_0 = 0) = 1$ is optimal in yielding the minimal expected free energy.
\item[(iii)] The two previous points imply that the chain's empirical measure must converge to \textit{energy-minimizing} stationarity.
This fact provides a variational formula \eqref{variational_formula} for the limiting free energy.
Furthermore, by examining what properties an energy-minimizing stationary distribution must have, we can deduce certain asymptotic properties of the chain.
In particular, Theorems \ref{apa_new} and \ref{geometric_new} will follow, as described in Sections \ref{varitional_formula_section} and \ref{proof_main_results}.
\end{itemize}
Two of the most challenging steps are (a) constructing the compact space; and (b) proving continuity of $\t$.
Although the remaining work to prove Theorems \ref{apa_new} and \ref{geometric_new} can largely be taken from \cite{bates-chatterjee17}, these two parts must undergo non-trivial generalizations.
This is done in Section \ref{adaptation_update_map}.

First, the construction in \cite{bates-chatterjee17} of the compact metric space $(\s,d_*)$ depended on the assumption $\lambda(\pm2\beta) < \infty$.
Now working under the more natural hypothesis \eqref{mgf_assumption}, we must define a one-parameter family of metrics $(d_\alpha)_{\alpha > 1}$, whereby $\alpha$ can be chosen so that $\alpha\beta < \beta_{\max}$.
For each $\alpha > 1$, it must be checked that $d_\alpha$ is a metric, and that it induces a compact topology.

Second, the definition of $\t$ depends very explicitly on $P$, the law of the reference walk.
When this walk is SRW, showing continuity of $\t$ amounts to proving that local interactions between endpoints are modified continuously with the addition of another monomer.
When $P$ is general, however, there can be interactions of endpoints arbitrarily far apart.
Showing that these interactions do not spoil continuity complicates the proof of Proposition \ref{continuity}.
For the same reason, technical details become more difficult in the proof of Proposition \ref{phase_transition}, which is the generalization of Theorem \ref{phase_transition_background}.

While this paper is focused on polymers in $\Z^d$, the above program could be carried out on other countable, locally finite Cayley graphs.
In this setting, the translation action on $\Z^d$ is generalized by the group's natural action on itself.
One case of interest is the infinite binary tree, considered as a subset of the canonical Cayley graph for the free group on two generators.
An intriguing observation is that while analogs of Theorems \ref{apa_new} and \ref{geometric_new} should go through, the ``favorite region conjecture" mentioned in Section \ref{results_this_paper} will not.
Indeed, when the binary tree having $2^n$ leaves is identified with $[0,1]$ having $2^n$ dyadic subintervals, the limiting endpoint distribution of a low-temperature tree polymer converges in law to a purely atomic measure \cite{barral-rhodes-vargas12}, in analogy with Theorem \ref{apa_new}.
While this measure sometimes confines almost all of its mass to a very narrow interval---in analogy with Theorem \ref{geometric_new}---the low probability of this event prevents a stronger localization result.
For more on the comparison of tree polymers versus lattice polymers, see \cite[Chapter 4]{comets17}.

\subsection{Related aspects of short- and long-range polymers} \label{related_aspects}

\subsubsection{Variational formulas for free energy} \label{intro_variational_formulas}
As in other statistical mechanical models, there is great interest in computing the limiting free energy $p(\beta)$ from Theorem \ref{convergence_background}.
For $P$ having finite support (i.e.~$P(y,x) > 0$ for only finitely many $x$), a series of papers due to Georgiou, Rassoul-Agha, Sepp\"al\"ainen, and Yilmaz \cite{rassoul-seppalainen-yilmaz13,rassoul-seppalainen-yilmaz17,georgiou-rassoul-seppalainen16} provides two types of variational formulas for $p(\beta)$, one using cocycles (additive functions on $\N \times \Z^d$) and another of a Gibbs variational form (optimizing the balance of energy versus entropy).
A third type of variational formula appeared in \cite{bates-chatterjee17} for the SRW case, and is extended to the general case in this paper as \eqref{variational_formula}.
This formula arises naturally as an optimization over a functional order parameter---in this case, the law of the endpoint distribution---following a variant of the cavity method used to study spin glass systems, in particular the Sherrington-Kirkpatrick model \cite{talagrand11I,talagrand11II,panchenko13}.
As such, the variational formula \eqref{variational_formula} can be considered in analogy with the  Parisi formula \cite{talagrand06}, which has recently been the object of intense study (e.g.~see \cite{panchenko05,panchenko08,panchenko13II,chen13,panchenko14,auffinger-chen15I,auffinger-chen15II,auffinger-chen16,jagannath-tobasco16,auffinger-chen17,chen17}).
For further comparison, see \cite[Section 1.4.3]{bates-chatterjee17}.

\subsubsection{Martingale phase transition and free energy asymptotics}
It is not difficult to check that the normalized partition function $W_n \coloneqq Z_n/\EE(Z_n)$ forms a positive martingale with respect to $\f_n$.
Therefore, the martingale convergence theorem guarantees the existence of a random variable $W_\infty \geq 0$ such that $W_n \to W_\infty$ almost surely, while Kolmogorov's $0$-$1$ law implies $\PP(W_\infty > 0) \in \{0,1\}$.
It is known in the SRW case \cite[Theorem 3.2(a)]{comets-yoshida06} and in the long-range case (stated but not proved in \cite[Theorem 1.4]{wei16}) that there is a phase transition.
That is, there exists $\bar\beta_c \in [0,\infty]$ such that 
\eq{
\beta < \bar\beta_c \quad &\Rightarrow \quad \PP(W_\infty > 0) = 1 \qquad \text{(``weak disorder")} \\
\beta > \bar\beta_c \quad &\Rightarrow \quad \PP(W_\infty = 0) = 1 \qquad \text{(``strong disorder").}
}
Since $W_n \to 0$ exponentially quickly when $\beta > \beta_c$, it is clear that $\bar\beta_c \leq \beta_c$, although it is conjectured that $\bar\beta_c = \beta_c$ in general.\footnote{This is well-known for analogous models on $b$-ary trees.
A nice exposition can be found in \cite[Chapter 4]{comets17}.}
This is known only in exceptional (but highly non-trivial) cases when $\bar \beta_c = 0$.
For the SRW case, it was proved that $\beta_c = 0$ by Comets and Vargas for $d = 1$ \cite[Theorem 1.1]{comets-vargas06} and by Lacoin for $d = 2$ \cite[Theorem 1.6]{lacoin10}.
For the long-range $\alpha$-stable case with $\alpha \in (1,2]$ and $d = 1$, Wei showed the analogous result: $\beta_c = \bar \beta_c = 0$ \cite[Theorem 1.11(ii)]{wei16}.\footnote{On the other hand, for $d \geq 3$ and any truly $d$-dimensional random walk $P$, the transience of $P$ forces $\bar\beta_c > 0$.
Also when $d = 1$ and $\alpha \in (0,1)$, or $d = 2$ and $\alpha \in (0,2)$, transience implies $\bar\beta_c > 0$, see \cite[Theorem 4.1]{comets07}.}
Recently, Wei showed \cite[Theorem 1.3]{wei18} that in the critical case $\alpha = d = 1$ and under some regularity assumptions on $P$, $\bar\beta_c = 0$ again implies $\beta_c = 0$.

The results just mentioned from \cite{lacoin10,wei16,wei18} are in fact corollaries of asymptotics obtained for $p(\beta)$ as $\beta \searrow 0$.
In the SRW case, the bounds in \cite{lacoin10} have been subsequently sharpened \cite{watbled12,nakashima14,alexander-yildirim15,berger-lacoin17}.
For $d=1$, the precise exponent seen in these results is related to the scaling of $\beta_n \searrow 0$ that generates the intermediate disorder regime in which a rescaled lattice polymer converges to the continuum random polymer \cite{alberts-khanin-quastel14I}.
This method of identifying a KPZ regime was initiated in \cite{alberts-khanin-quastel14II} and extended to certain long-range cases in \cite{caravenna-sun-zygouras17I,caravenna-sun-zygouras17II}.
In fact, the authors of \cite{caravenna-sun-zygouras17II} are able to identify a universal limit for the point-to-point log partition functions, in \textit{critical} cases, for both $d= 1$ and $d=2$.
Related work on the stochastic heat equation has been done for $d \geq 3$ \cite{mukherjee-shamov-zeitouni16}.
Finally, in \cite{comets-fukushima-nakajima-yoshida15} the asymptotics of $p(\beta)$ as $\beta\to\pm\infty$ are derived when $P$ has a stretched exponential tail, and the environment consists of Bernoulli random variables.

\subsubsection{Continuous versions of long-range polymers}
In \cite{miura-tawara-tsuchida08}, a continuous version of $\alpha$-stable long-range polymers is considered.
Specifically, a phase transition was shown for the normalized partition function associated to a L\'evy process subjected to a Poissonian random environment. 
Sufficient conditions were given for either side of the transition.
In the same way that the $\alpha$-stable polymer introduced in \cite{comets07} generalized classical lattice polymers, this L\'evy process model generalized a Brownian motion in Poissonian environment, which was introduced in \cite{comets-yoshida05} and also considered in \cite{lacoin11,comets-yoshida13}.

\section{Free energy and phase transition} \label{free_energy_background}
In order to give context for the main results, which concern the behavior of polymer measures above and below a phase transition, we must first check that such a phase transition exists!
To do so, we need to prove Theorems \ref{convergence_background} and \ref{phase_transition_background} in the general setting.
First, in Section \ref{convergence_of_free_energy} we show that the quenched free energy has a deterministic limit. 
The arguments used here are standard, and the expert reader may skip them; nevertheless, the details are included to verify that no essential facts are lost when working with an arbitrary reference walk.
Next, a proof of the phase transition is given in Section \ref{existence_of_phase_transition}.
In particular, the methods initiated in \cite{comets-yoshida06} must be refined to account for general $P$ and weaker assumptions on the logarithmic moment generating function $\lambda(\beta)$.

\subsection{Convergence of free energy} \label{convergence_of_free_energy}
We begin by showing the existence of a limiting free energy.
Throughout this section one may assume a condition just slightly weaker than \eqref{mgf_assumption}, namely
\eeq{ \label{weaker_mgf_assumption}
\lambda(\pm\beta) < \infty.
}

\begin{prop} \label{free_energy_converges}
Assume \eqref{weaker_mgf_assumption}.
Then the limiting free energy exists and is deterministic:
\eq{
\lim_{n \to \infty} F_n = p(\beta) \quad \mathrm{a.s.} \text{ and in } L^p,\, p \in [1,\infty).
}
\end{prop}
The proof follows the usual program of showing first that $\EE(F_n)$ converges, and second that $F_n$ concentrates around its mean.

\begin{lemma} \label{means_converge}
Assume \eqref{weaker_mgf_assumption}.
Then
\eeq{
\lim_{n \to \infty} \EE(F_n) = \sup_{n \geq 0} \EE(F_n) \eqqcolon p(\beta) \leq \lambda(\beta). \label{p_def}
}
\end{lemma}

\begin{lemma}[{cf.~\cite[Theorem 1.4]{liu-watbled09}}] \label{concentration}
Assume \eqref{weaker_mgf_assumption}.
Then
\eq{ 
\PP(|F_n - \EE(F_n)| > x) \leq \begin{cases}
2e^{-ncx^2} &\mathrm{if~} x \in [0,1] \\
2e^{-ncx} &\mathrm{if~} x > 1, \end{cases}
}
where $c > 0$ is a constant depending only on the value of $K \coloneqq 2e^{\lambda(\beta)+\lambda(-\beta)}$.
In particular,
\eq{ 
\lim_{n \to \infty} |F_n - \EE(F_n)| = 0  \quad \mathrm{a.s.~and~in~} L^p,\, p \in [1,\infty).
}
\end{lemma}

Proposition \ref{free_energy_converges} is immediate given Lemmas \ref{means_converge} and \ref{concentration}, and the proof from \cite{liu-watbled09} of Lemma \ref{concentration} requires no modification for general $P$. 
So we will just prove Lemma \ref{means_converge}.

\begin{proof}[Proof of Lemma \ref{means_converge}]
The equality in \eqref{p_def} follows from Fekete's lemma, once we show that the sequence $(\EE \log Z_n)_{n\geq0}$ is superadditive.
Indeed, we will simply generalize the argument seen, for instance, in \cite[p.~440]{carmona-hu02}.
But first we check that $\EE(\log Z_n)$ is finite for each $n$.
To obtain an upper bound, we use Jensen's inequality applied to $t \mapsto \log t$.
This just yields the annealed bound,
\eq{
\EE(\log Z_n) \leq \log \EE(Z_n) = \log \EE\bigg[\sum_{x_1,\dots,x_n} \exp\bigg(\beta\sum_{i=1}^n\eta(i,x_i)\bigg)\prod_{i=1}^nP(x_{i-1},x_i)\bigg].
}
The nonnegativity of all summands allows us, by Tonelli's theorem, to pass the expectation through the sum.
That is,
\eq{
\EE(Z_n) = \sum_{x_1,\dots,x_n} \EE\exp\bigg(\beta\sum_{i=1}^n\eta(i,x_i)\bigg)\prod_{i=1}^nP(x_{i-1},x_i)
= \sum_{x_1,\dots,x_n} e^{n\lambda(\beta)}\prod_{i=1}^nP(x_{i-1},x_i) = e^{n\lambda(\beta)},
}
and so
\eq{
\EE(\log Z_n) \leq n\lambda(\beta) < \infty.
}
In particular, the inequality in \eqref{p_def} follows from the above display.

On the other hand, a lower bound is found by again using Jensen's inequality, but applied to $t \mapsto e^t$ and with respect to $P$:
\eq{
\EE(\log Z_n) &= \EE\bigg[\log \sum_{x_1,\dots,x_n} \exp\bigg(\beta\sum_{i=1}^n\eta(i,x_i)\bigg)\prod_{i=1}^nP(x_{i-1},x_i)\bigg] \\
&\geq \EE\bigg[\sum_{x_1,\dots,x_n} \beta \sum_{i = 1}^n \eta(i,x_i) \prod_{i = 1}^n P(x_{i-1},x_i)\bigg]
= \beta\EE(\eta),
}
where the final equality is a consequence of Fubini's theorem.
Indeed, we may apply Fubini's theorem because
\eq{
\EE|\eta| 
= \int_0^\infty \PP(|\eta|\geq t)\ \dd t
\leq \int_0^\infty e^{-\beta t} \EE(e^{\beta|\eta|})\ \dd t
< \infty.
}
Therefore, we obtain the desired lower bound:
\eq{
\EE(\log Z_n) \geq \beta\EE(\eta) > -\infty.
}

Now we may prove superadditivity.
For a given integer $k \geq 0$ and $y \in \Z^d$, let $\theta_{k,y}$ be the associated time-space translation of the environment:
\eq{
(\theta_{k,y}\, \eta)(i,x) = \eta(i+k,x+y).
}
Because the collection $(\eta(i,x): i \geq 1, x\in\Z^d)$ is i.i.d., the random variables $Z_n$ and $Z_n \circ \theta_{k,y}$ have the same law.
Furthermore, for any $0 \leq k \leq n$, we have the identity
\eeq{
Z_{n} = \sum_{y \in \Z^d} Z_{k}(y)\cdot(Z_{n-k} \circ \theta_{k,y}), \label{Z_identity}
}
where $Z_k(y) \coloneqq E(e^{-\beta H_k(\omega)}; \omega_k = y)$ is the contribution to $Z_k$ coming from the endpoint $y$.
By Jensen's inequality,
\eq{
\log Z_{n} = \log \sum_{y \in \Z^d} Z_{k}(y)\cdot(Z_{n-k} \circ \theta_{k,y})
&= \log Z_{k} + \log \sum_{y \in \Z^d} \frac{Z_{k}(y)}{Z_{k}}\cdot (Z_{n-k} \circ \theta_{k,y}) \\
&\geq \log Z_{k} + \sum_{y \in \Z^d} \frac{Z_{k}(y)}{Z_{k}}\log(Z_{n-k} \circ \theta_{k,y}).
}
Notice that $Z_{n-k} \circ \theta_{k,y}$ depends only on the environment after time $k$, meaning it is independent of $\f_k$. 
We thus have
\eq{
\EE\givenk[\bigg]{\sum_{y \in \Z^d} \frac{Z_{k}(y)}{Z_{k}}\log(Z_{n-k} \circ \theta_{k,y})}{\f_k} 
&= \sum_{y \in \Z^d} \frac{Z_{k}(y)}{Z_{k}}\EE(\log Z_{n-k} \circ \theta_{k,y})
= \EE(\log Z_{n-k}).
}
Using this observation in the previous display, we arrive at the desired superadditive inequality:
\eq{
\EE(\log Z_{n}) \geq \EE(\log Z_{k}) + \EE(\log Z_{n-k}).
}
\end{proof}

\subsection{Existence of critical temperature} \label{existence_of_phase_transition}
Now we prove a phase transition between the high temperature and low temperature regimes.

\begin{prop} \label{phase_transition}
Assume
\eq{
\beta_{\max} \coloneqq \sup\{\beta \geq 0 : \lambda(\pm\beta) < \infty\} \in (0,\infty].
}
Then $\lambda - p \geq 0$ is non-decreasing on the interval $[0,\beta_{\max})$.
In particular, there exists a critical value $\beta_c = \beta_c(d,\mathfrak{L}_\eta,P) \in [0,\beta_{\max}]$ such that for every $\beta \in (0,\beta_{\max})$,
\begin{align}
0 \leq \beta \leq \beta_c \quad &\Rightarrow \quad p(\beta) = \lambda(\beta) \label{below_phase_transition}\\
\beta > \beta_c \quad &\Rightarrow \quad p(\beta) < \lambda(\beta). \label{above_phase_transition}
\end{align}
\end{prop}

\begin{remark}
Before the proof, some comments are in order: \begin{itemize}
\item This phase transition was proved by Comets and Yoshida \cite[Theorem 3.2(b)]{comets-yoshida06} for the SRW case under the hypothesis $\beta_{\max} = \infty$, and we adopt their general proof strategy.
\item Two generalizations were claimed to follow easily from the same methods.
Vargas \cite[Lemma 3.4]{vargas07} suggests the hypothesis $\beta_{\max} = \infty$ can be dropped, while Comets \cite[Theorem 6.1]{comets07} allows $P$ to be $\alpha$-stable, $\alpha \in [0,2)$.
We do both by assuming only $\beta_{\max} > 0$ and allowing $P$ to be a general random walk.
It seems the resulting difficulties are only technical, but how to resolve them is not obvious, and so we provide a full proof.
\item If $\lim_{\beta\to\beta_{\max}}\beta\lambda'(\beta) - \lambda(\beta) = \infty$, then Theorem \ref{localization_sufficient_background} guarantees $\beta_c < \beta_{\max}$ whenever the entropy of $P(\omega_1 = \cdot)$ is finite.
\end{itemize}
\end{remark}

\begin{lemma} \label{decomposition}
For any fixed $\beta > 0$, there is a decomposition
\eq{
|x|e^{\beta x} = g(x) - h(x),
}
where $g : \R \to \R$ is non-decreasing, and $0 \leq h(x) \leq (\beta e)^{-1}$ for all $x \in \R$.
\end{lemma}

\begin{proof}
Observe that $\vphi(x) \coloneqq |x|e^{\beta x}$ is an increasing function on $(-\infty,-\beta^{-1}]$ and on $[0,\infty)$, with
$\vphi(-\beta^{-1}) = (\beta e)^{-1}$ and $\vphi(x_0) = (\beta e)^{-1}$ for some (unique) $x_0 > 0$.
Therefore, we define
\eq{
g(x) \coloneqq \begin{cases} (\beta e)^{-1} &\text{if }x \in [-\beta^{-1},x_0] \\
|x|e^{\beta x} &\text{otherwise}, \end{cases} \qquad
h(x) \coloneqq g(x) - |x|e^{\beta x}.
}
Then $g$ is non-decreasing with $g \geq \vphi$, and so
$0 \leq h(x) \leq (\beta e)^{-1}$.
\end{proof}

\begin{proof}[Proof of Proposition \ref{phase_transition}]
Note that
$\log Z_n = E(e^{-\beta H_n(\omega)})$
is the (random) logarithmic moment generating function for $-H_n(\omega)$ with respect to $P$, and is finite for all $\beta \in [0,\beta_{\max})$ almost surely by the proof of Lemma \ref{means_converge}.
Therefore, $\EE(\log Z_n)$ is convex on $(0,\beta_{\max})$.
Now seen to be the limit of convex functions, $p$ must be convex on $(0,\beta_{\max})$.
It follows from general convex function theory that $p$ is differentiable almost everywhere on $(0,\beta_{\max})$, and
\eq{
p'(\beta) = \lim_{n \to \infty} \frac{\partial}{\partial\beta} \EE(F_n) \quad \text{whenever $p'(\beta)$ exists}.
}
Furthermore, convexity implies $p$ is absolutely continuous on any closed subinterval of $[0,\beta_{\max})$, and thus
\eq{
p(\beta_0) = \int_0^{\beta_0} p'(\beta)\ \dd\beta \quad \text{for all } \beta_0 \in (0,\beta_{\max}).
}
Suppose we can show
\eeq{ \label{derivative_ineq_to_show}
\frac{\partial}{\partial \beta}\EE(F_n) \leq \lambda'(\beta) \quad \text{for all $\beta \in (0,\beta_{\max})$.}
}
Then we could conclude that
\eq{
\lambda(\beta_0) - p(\beta_0) = \int_0^{\beta_0} [\lambda'(\beta)-p'(\beta)]\ \dd\beta
= \int_0^{\beta_0} \Big[\lambda'(\beta) - \lim_{n\to\infty} \frac{\partial}{\partial\beta}\EE(F_n)\Big]\ \dd\beta
}
is an increasing function of $\beta_0 \in [0,\beta_{\max})$, with
$p(0) = \lambda(0) = 0.$
In particular, the existence of $\beta_c$ will be proved.
Therefore, we need only to show \eqref{derivative_ineq_to_show}.
To do so, we would like to write
\eeq{ \label{would_like}
\frac{\partial}{\partial\beta} \EE(\log Z_n)
\stackrel{\text{(a)}}{=} \EE\Big[\frac{\partial}{\partial\beta}\log Z_n\Big]
= \EE\bigg[\frac{\frac{\partial}{\partial\beta}Z_n}{Z_n}\bigg] 
&\stackrel{\phantom{\text{(b)}}}{=} \EE\bigg[\frac{\frac{\partial}{\partial\beta}E(e^{-\beta H_n(\omega)})}{Z_n}\bigg] \\
&\stackrel{\text{(b)}}{=} \EE\bigg[E\bigg(\frac{-H_n(\omega)e^{-\beta H_n(\omega)}}{Z_n}\bigg)\bigg] \\
&\stackrel{\text{(c)}}{=} E\bigg[\EE\bigg(\frac{-H_n(\omega)e^{-\beta H_n(\omega)}}{Z_n}\bigg)\bigg], \\
}
but each of (a), (b), and (c) require justification.
Postponing these technical verifications for the moment, we complete the proof of \eqref{derivative_ineq_to_show} assuming \eqref{would_like}.

For a fixed $\omega \in \Omega_p$, we can write
\eq{
\EE\bigg(\frac{-H_n(\omega)e^{-\beta H_n(\omega)}}{Z_n}\bigg) = e^{n\lambda(\beta)}\wt\EE\bigg(\frac{-H_n(\omega)}{Z_n}\bigg),
}
where $\wt\EE$ denotes expectation with respect to the probability measure $\wt\PP$ given by
\eq{
\dd\wt\PP = \frac{e^{-\beta H_n(\omega)}}{e^{n\lambda(\beta)}}\ \dd\PP
= \frac{e^{-\beta H_n(\omega)}}{\EE(e^{-\beta H_n(\omega)})}\ \dd\PP.
}
Since the Radon-Nikodym derivative
\eq{
\frac{\dd \wt\PP}{\dd\PP} = e^{-\beta H_n(\omega)-n\lambda(\beta)} = \prod_{i = 1}^n e^{\beta \eta(i,\omega_i)-\lambda(\beta)}
}
is a product of independent quantities (with respect to $\PP$), the probability measure $\wt\PP$ remains a product measure.
Therefore, we can apply the Harris--FKG inequality (e.g.~see \cite[Theorem 2.4]{grimmett99}):
$-H_n(\omega)$ is non-decreasing in all $\eta(i,x)$, while $Z_n^{-1}$ is decreasing, which implies
\eq{
\wt\EE\bigg(\frac{-H_n(\omega)}{Z_n}\bigg) &\leq 
\wt\EE(-H_n(\omega))\wt\EE(Z_n^{-1}),
}
where
\eq{
\wt\EE(-H_n(\omega)) &= e^{-n\lambda(\beta)} \EE(-H_n(\omega)e^{-\beta H_n(\omega)}) \\
&= e^{-n\lambda(\beta)} \sum_{i = 1}^n \EE(\eta(i,\omega_i)e^{-\beta H_n(\omega)}) \\
&= e^{-n\lambda(\beta)} \sum_{i = 1}^n \underbrace{\EE(\eta(i,\omega_i)e^{\beta \eta(i,\omega_i)})}_{\lambda'(\beta)e^{\lambda(\beta)}}\prod_{j \neq i} \underbrace{\EE(e^{\beta\eta(j,\omega_j)})}_{e^{\lambda(\beta)}}
= n\lambda'(\beta).
}
Therefore,
\eeq{ \label{last_inequality}
\EE\bigg(\frac{-H_n(\omega)e^{-\beta H_n(\omega)}}{Z_n}\bigg) \leq e^{n\lambda(\beta)}\wt\EE(-H_n(\omega))\wt\EE(Z_n^{-1}) = n\lambda'(\beta)\cdot\EE(Z_n^{-1}e^{-\beta H_n(\omega)}).
}
We now have
\eq{
\frac{\partial}{\partial \beta} \EE(\log Z_n) 
\stackrel{\mbox{\scriptsize\eqref{would_like}}}{=} E\bigg[\EE\bigg(\frac{-H_n(\omega)e^{-\beta H_n(\omega)}}{Z_n}\bigg)\bigg]
&\stackrel{\mbox{\scriptsize\eqref{last_inequality}}}{\leq} n\lambda'(\beta)\cdot E[\EE(Z_n^{-1}e^{-\beta H_n(\omega)})] \\
&\stackrel{\phantom{\eqref{last_inequality}}}{=} n\lambda'(\beta)\cdot\EE[Z_n^{-1}E(e^{-\beta H_n(\omega)})] = n\lambda'(\beta),
}
where the penultimate equality is a consequence of Tonelli's theorem, since $Z_n^{-1}e^{-\beta H_n(\omega)} > 0$.
The inequality \eqref{derivative_ineq_to_show} now follows by dividing by $n$.

\subsubsection{Justification of $\mathrm{(c)}$ in \eqref{would_like}}
Fix $\beta \in (0,\beta_{\max})$.
Choose $q > 1$ such that $q\beta < \beta_{\max}$, and let $q'$ be its H\"{o}lder conjugate:
$1/q + 1/q' = 1$.
Step (c) in \eqref{would_like} will follow from Fubini's theorem once we verify that
\eeq{ \label{before_expectation}
E\bigg(\EE\bigg|\frac{-H_n(\omega)e^{-\beta H_n(\omega)}}{Z_n}\bigg|\bigg) < \infty.
}
Let $g$ and $h$ be as in Lemma \ref{decomposition}, so that we can write
\eeq{ \label{g_and_h}
\bigg|\frac{-H_n(\omega)e^{-\beta H_n(\omega)}}{Z_n}\bigg|
= \frac{|-H_n(\omega)|e^{-\beta H_n(\omega)}}{Z_n}
= \frac{g(-H_n(\omega))}{Z_n} - \frac{h(-H_n(\omega))}{Z_n} \leq \frac{g(-H_n(\omega))}{Z_n}.
}
Temporarily fix a path $\omega \in \Omega_p$.
Since $Z_n$ and $-H_n$, and therefore $g(-H_n)$, are non-decreasing functions of all $\eta(i,x)$, the Harris--FKG inequality shows
\eeq{ \label{fkg_1}
\EE\Big(\frac{g(-H_n(\omega))}{Z_n}\Big) \leq \EE\big[g(-H_n(\omega))\big]\EE(Z_n^{-1}).
}
The first factor satisfies
\eeq{ \label{first_factor}
\EE\big[g(-H_n(\omega))\big]
&= \EE\big[|-H_n(\omega)|e^{-\beta H_n(\omega)}\big] + \EE\big[h(-H_n(\omega))\big] \\
&\leq \EE\bigg[\sum_{i = 1}^n |\eta(i,\omega_i)|\exp\bigg(\beta\sum_{i = 1}^n \eta(i,\omega_i)\bigg)\bigg] + (\beta e)^{-1} \\
&= \EE\bigg[\sum_{i=1}^n |\eta(i,\omega_i)|e^{\beta \eta(i,\omega_i)}\prod_{j\neq i} e^{\beta\eta(j,\omega_j)}\bigg] + (\beta e)^{-1} \\
&\leq n\EE(|\eta| e^{\beta \eta})e^{(n-1)\lambda(\beta)} + (\beta e)^{-1} \\
&\leq n(\EE|\eta|^{q'})^{1/q'}e^{\lambda(q\beta)/q}e^{(n-1)\lambda(\beta)} + (\beta e)^{-1} < \infty.
}
The second factor satisfies
\eeq{ \label{second_factor}
\EE(Z_n^{-1}) = \EE\big(E(e^{-\beta H_n(\omega)})^{-1}\big) \leq \EE\big(E(e^{\beta H_n(\omega)})\big) = E\big(\EE(e^{\beta H_n(\omega)})\big) = n\lambda(-\beta) < \infty,
}
where we have used Tonelli's theorem to exchange the order of integration.
We have thus shown
\eeq{ \label{fubini_hypothesis}
E\bigg(\EE\bigg|\frac{-H_n(\omega)e^{-\beta H_n(\omega)}}{Z_n}\bigg|\bigg)
\stackrel{\mbox{\scriptsize\eqref{g_and_h}}}{\leq} E\bigg(\EE\Big[\frac{g(-H_n(\omega))}{Z_n}\Big]\bigg)
\stackrel{\text{\eqref{fkg_1}--\eqref{second_factor}}}{<} \infty,
}
as desired.

\subsubsection{Justification of $\mathrm{(b)}$ in \eqref{would_like}} \label{justification_b}
By simple differentiation rules,
\eq{
\frac{\partial}{\partial\beta}\log Z_n = \frac{\frac{\partial}{\partial\beta}Z_n}{Z_n}
= \frac{\frac{\partial}{\partial \beta} E(e^{-\beta H_n(\omega)})}{Z_n}.
}
We would like to pass the derivative through the expectation and write
\eq{
\frac{\partial}{\partial \beta} E(e^{-\beta H_n(\omega)}) = E(-H_n(\omega)e^{-\beta H_n(\omega)}).
}
That is, we claim
\eeq{ \label{derivative_through_EE}
&\frac{\partial}{\partial \beta}\bigg[\sum_{x_1,x_2,\dots,x_n \in \Z^d} \exp\bigg(\beta\sum_{i = 1}^n \eta(i,x_i)\bigg)\prod_{i=1}^n P(x_{i-1},x_i)\bigg] \\
&= \sum_{x_1,x_2,\dots,x_n \in \Z^d} \bigg(\sum_{i = 1}^n \eta(i,x_i)\bigg)\exp\bigg(\beta\sum_{i = 1}^n \eta(i,x_i)\bigg)\prod_{i=1}^n P(x_{i-1},x_i) \quad \mathrm{a.s.}
}
To show \eqref{derivative_through_EE} at a particular $\beta_0 \in (0,\beta_{\max})$, it suffices to exhibit $\eps > 0$ and a constant $C_{\vc\eta} < \infty$ depending only on the quenched environment $\vc\eta$, such that
\eeq{ \label{suffices_for_need}
\sum_{x_1,x_2,\dots,x_n\in\Z^d} \bigg|\sum_{i = 1}^n \eta(i,x_i)\bigg|\exp\bigg(\beta\sum_{i = 1}^n \eta(i,x_i)\bigg)\prod_{i=1}^n P(x_{i-1},x_i) < C_{\vc\eta} \quad \forall\, \beta \in [\beta_0-\eps,\beta_0+\eps].
}
Indeed, if \eqref{suffices_for_need} holds, then for $\beta \in [\beta_0-\eps,\beta_0+\eps]$,
\eq{
\sum_{\substack{x_1,x_2,\dots,x_n \in \Z^d \\ \|x_i\|_1 \leq M\, \forall\, i}} \bigg(\sum_{i = 1}^n \eta(i,x_i)\bigg)\exp\bigg(\beta\sum_{i = 1}^n \eta(i,x_i)\bigg)\prod_{i=1}^n P(x_{i-1},x_i)
\xrightarrow[{M\to\infty}]{\mathrm{uniformly~in~}\beta} E(-H_n(\omega)e^{-\beta H_n(\omega)}).
}
In summary, we know by definition that
\eq{
\sum_{\substack{x_1,x_2,\dots,x_n \in \Z^d \\ \|x_i\|_1 \leq M\, \forall\, i}} \exp\bigg(\beta\sum_{i = 1}^n \eta(i,x_i)\bigg)\prod_{i=1}^n P(x_{i-1},x_i) \xrightarrow[M\to\infty]{} E(e^{-\beta H_n(\omega)}) \quad
\text{for all $\beta \in [0,\beta_{\max})$},
}
and we know by the above argument that the derivative of the left-hand side converges uniformly to $E(-H_n(\omega)e^{-\beta H_n(\omega)})$ near $\beta_0$.
It follows that
\eeq{ \label{derivative_on_interval}
\frac{\partial}{\partial\beta}E(e^{-\beta H_n(\omega)})
= E(-H_n(\omega)e^{-\beta H_n(\omega)}) \quad \text{for all $\beta \in (\beta_0-\eps,\beta_0+\eps)$,}
}
in particular when $\beta = \beta_0$.
We are thus left only with the task of establishing \eqref{suffices_for_need} for almost every $\vc\eta$.

Fix $\beta_0 \in (0,\beta_{\max})$, and let $\eps > 0$ be such that $\beta_0 + \eps < \beta_{\max}$ and $\beta_0 - \eps > 0$.
Choose a number $q > 1$ such that $q(\beta_0+\eps) < \beta_{\max}$, and let $q' > 1$ denote its H\"{o}lder conjugate: $1/q + 1/q' = 1$.
For all $\beta \in [\beta_0-\eps,\beta_0+\eps]$, we have the uniform upper bound
\eq{
&\sum_{x_1,x_2,\dots,x_n\in\Z^d} \bigg|\sum_{i = 1}^n \eta(i,x_i)\bigg|\exp\bigg(\beta\sum_{i = 1}^n \eta(i,x_i)\bigg)\prod_{i=1}^n P(x_{i-1},x_i) \\
&\leq \sum_{x_1,x_2,\dots,x_n\in\Z^d} \sum_{i = 1}^n |\eta(i,x_i)|\exp\bigg((\beta_0+\eps)\sum_{i = 1}^n \eta(i,x_i)_+\bigg) \prod_{i=1}^n P(x_{i-1},x_i) \\
&\leq \sum_{x_1,x_2,\dots,x_n\in\Z^d} \bigg[\sum_{i = 1}^n |\eta(i,x_i)|(e^{(\beta_0+\eps)\eta(i,x_i)}+1)\prod_{j \neq i} (e^{(\beta_0+\eps)\eta(j,x_j)}+1)\bigg] \prod_{i=1}^n P(x_{i-1},x_i),
}
where $x_+ \coloneqq \max(0,x)$.
Furthermore, this upper bound is finite for almost every $\vc\eta$, since
\eq{
&\sum_{x_1,x_2,\dots,x_n\in\Z^d} \EE\bigg[\sum_{i = 1}^n |\eta(i,x_i)|(e^{(\beta_0+\eps)\eta(i,x_i)}+1)\prod_{j \neq i} (e^{(\beta_0+\eps)\eta(j,x_j)}+1)\bigg] \prod_{i=1}^n P(x_{i-1},x_i) \\
&= n\EE\big[|\eta|(e^{(\beta_0+\eps)\eta}+1)\big]\big(\EE[e^{(\beta_0+\eps)\eta}+1]\big)^{n-1} \\
&\leq n(\EE|\eta|^{q'})^{1/q'}\Big(\EE\big[(e^{(\beta_0+\eps)\eta}+1)^q\big]\Big)^{1/q}\big(\EE[e^{(\beta_0+\eps)\eta}+1]\big)^{n-1} \\
&\leq n(\EE|\eta|^{q'})^{1/q'}\Big(2^q\EE\big[e^{q(\beta_0+\eps)\eta}+1\big]\Big)^{1/q}\big(\EE[e^{(\beta_0+\eps)\eta}+1]\big)^{n-1} \\
&= 2n(\EE|\eta|^{q'})^{1/q'}\big(e^{\lambda(q(\beta_0+\eps))}+1\big)^{1/q}(e^{\lambda(\beta_0+\eps)}+1)^{n-1} < \infty.
}
In particular, some $C_{\vc\eta}$ satisfying \eqref{suffices_for_need} exists almost surely, and so
\eqref{derivative_through_EE} holds almost surely.
Therefore, (b) holds in \eqref{would_like}:
\eq{
\EE\bigg[\frac{\frac{\partial}{\partial\beta}E(e^{-\beta H_n(\omega)})}{Z_n}\bigg]
= \EE\bigg[\frac{E(-H_n(\omega)e^{-\beta H_n(\omega)})}{Z_n}\bigg]
=  \EE\bigg[E\bigg(\frac{-H_n(\omega)e^{-\beta H_n(\omega)}}{Z_n}\bigg)\bigg].
}

\subsubsection{Justification of $\mathrm{(a)}$ in \eqref{would_like}}
Fix $\beta_0 \in (0,\beta_{\max})$ and $\eps > 0$ such that $\beta_0+\eps < \beta_{\max}$ and $\beta_0-\eps > 0$.
We begin by writing
\eq{
\Big[\frac{\partial}{\partial \beta} \EE(\log Z_n)\Big]_{\beta=\beta_0} = 
\lim_{h \to 0} \EE\Big[\frac{\log Z_n(\beta_0+h) - \log Z_n(\beta_0)}{h}\Big].
}
We will ultimately invoke dominated convergence to pull the limit through the expectation.
Consider any fixed $h$ satisfying $|h| \leq \eps$.
Now, $\log Z_n$ is almost surely continuously differentiable, and
so by the mean value theorem,
\eq{
\frac{\log Z_n(\beta_0+h) - \log Z_n(\beta_0)}{h} = \Big[\frac{\partial}{\partial\beta}\log Z_n\Big]_{\beta = \beta_0+\eps_{\vc\eta}} \quad \mathrm{a.s.}
}
for some $\eps_{\vc\eta}$ depending on the random environment $\vc\eta$ and satisfying $|\eps_{\vc\eta}| \leq |h| \leq \eps$.
But then by convexity of $\log Z_n$, we can bound the difference quotient by consider the endpoints of the interval $[\beta_0-\eps,\beta_0+\eps]$:
\eq{
\Big|\frac{\log Z_n(\beta_0+h) - \log Z_n(\beta_0)}{h}\Big| &= \Big|\Big[\frac{\partial}{\partial\beta}\log Z_n\Big]_{\beta=\beta_0+\eps_{\vc\eta}}\Big| \quad \mathrm{a.s.} \\
&\leq \max\bigg(\Big|\Big[\frac{\partial}{\partial\beta}\log Z_n\Big]_{\beta=\beta_0-\eps}\Big|,\Big|\Big[\frac{\partial}{\partial\beta}\log Z_n\Big]_{\beta=\beta_{0}+\eps}\Big|\bigg) \quad \mathrm{a.s.},
}
where the maximum is now a dominating function independent of $h$.
We showed \eqref{derivative_on_interval} holds almost surely, and so
\eq{
\frac{\partial}{\partial\beta}\log Z_n = E\bigg(\frac{-H_n(\omega)e^{-\beta H_n(\omega)}}{Z_n}\bigg)  \quad \text{for all $\beta \in (0,\beta_{\max})$} \quad \mathrm{a.s.}
}
Therefore, for any $\beta_1 \in (0,\beta_{\max})$ we have
\eq{
\EE\Big|\Big[\frac{\partial}{\partial\beta}\log Z_n\Big]_{\beta=\beta_1}\Big| = \EE\bigg|E\bigg(\frac{-H_n(\omega)e^{-\beta_1 H_n(\omega)}}{Z_n}\bigg)\bigg|
\leq \EE\bigg(E\bigg|\frac{-H_n(\omega)e^{-\beta_1 H_n(\omega)}}{Z_n}\bigg|\bigg)
\stackrel{\mbox{\scriptsize\eqref{fubini_hypothesis}}}{<} \infty.
}
In particular, the dominating function is integrable:
\eq{
&\EE\bigg[\max\bigg(\Big|\Big[\frac{\partial}{\partial\beta}\log Z_n\Big]_{\beta=\beta_0-\eps}\Big|,\Big|\Big[\frac{\partial}{\partial\beta}\log Z_n\Big]_{\beta=\beta_{0}+\eps}\Big|\bigg)\bigg] \\
&\leq \EE\Big|\Big[\frac{\partial}{\partial\beta}\log Z_n\Big]_{\beta=\beta_0 - \eps}\Big| + \EE\Big|\Big[\frac{\partial}{\partial\beta}\log Z_n\Big]_{\beta=\beta_{0}+\eps}\Big|
< \infty.
}
Dominated convergence now proves (a) in \eqref{would_like}:
\eq{
\Big[\frac{\partial}{\partial \beta} \EE(\log Z_n)\Big]_{\beta=\beta_0}
&= \lim_{h \to 0} \EE\Big[\frac{\log Z_n(\beta_0+h) - \log Z_n(\beta_0)}{h}\Big] \\
&= \EE\Big[\lim_{h \to 0}\frac{\log Z_n(\beta_0+h) - \log Z_n(\beta_0)}{h}\Big] 
= \EE\Big(\Big[\frac{\partial}{\partial\beta} \log Z_n\Big]_{\beta = \beta_0}\Big).
}
\end{proof}

\section{Adaptation of abstract machinery} \label{adaptation_update_map}
In this section we recall and adapt some necessary definitions and results from \cite{bates-chatterjee17}.
The key change we will make is to the definition of the ``update map" that sends a fixed endpoint distribution $\rho_n(\omega_n = \cdot)$ to the conditional law of $\rho_{n+1}(\omega_{n+1} = \cdot)$ given $\f_n$, viewed as a random variable in a suitable metric space $(\s,d)$.
The construction of $\s$ is identical to what was done in \cite{bates-chatterjee17}; we briefly describe it in Section \ref{pspm} and justify some adaptations in Section \ref{generalized_topology}.
Then, in Section \ref{update_map} we newly define the update map to allow for general reference walk $P$, and then lift it to a map of probability measures on $\s$, a space denoted $\p(\s)$.
Finally, in Section \ref{continuity} we prove continuity of the update map with respect to Wasserstein distance, which implies continuity of its lift.

\subsection{Partitioned subprobability measures} \label{pspm}
The \textit{quenched endpoint distribution} at time $n$, given by
\eeq{
f_n(x) \coloneqq \rho_n(\omega_n = x) = \frac{Z_n(x)}{Z_n}, \quad x \in \Z^d, \quad \text{where} \quad
Z_n(x) = E(e^{-\beta H_n(\omega)}; \omega_n = x),
 \label{fn_def}
}
is a Borel measurable function of $\vc\eta$ when considered in the space $\ell^1(\Z^d)$.
Of course, the distributions under consideration satisfy $f_n \geq 0$ and $\sum_{x\in\Z^d} f_n(x) = 1$ for all $n$.
Therefore, there is a natural embedding of these endpoint distributions into an auxiliary space
\eq{
\s_0 \coloneqq \{f : \N \times \Z^d \to [0,1] : \|f\| \leq 1\}, \quad \text{where} \quad \N = \{1,2,\dots\},
}
which inherits the $\ell^1(\N\times\Z^d)$ topology induced by the $\ell^1$-norm,
\eeq{
\|f\| \coloneqq \sum_{u \in \N \times \Z^d} |f(u)| \label{norm_def}.
}
For concreteness, the image of $f_n$ under this embedding can be explicitly defined by
\eq{
f_n(k,x) = \begin{cases}
\frac{Z_n(x)}{Z_n}  &\text{if $k = 1$} \\
0 &\text{otherwise}.
\end{cases}
}

For each $\alpha > 1$, we construct a pseudometric $d_\alpha$ on $\s_0$ as follows.
Define ``addition" and ``subtraction" on $\N \times \Z^d$ by extending the group structure of $\Z^d$, but only if the first coordinates agree:
\eq{
(n,x) \pm (m,y) = \begin{cases}
x \pm y &\text{if $n=m$} \\
\infty &\text{otherwise.}
\end{cases}
}
Similarly, define the $\ell^1$ ``norm" by
\eq{
\|(n,x)\|_1 = \|x\|_1, \qquad \|\infty\|_1 = \infty.
}
Now, for a finite subset $A \subset \N \times \Z^d$, call a map $\phi : A \to \N \times \Z^d$ an \textit{isometry} of degree $m$ if for every $u,v \in A$,
\eeq{
\|u - v\|_1 < m \quad \text{or} \quad \|\phi(u) - \phi(v)\|_1 < m \quad \Rightarrow \quad u - v = \phi(u) - \phi(v). \label{isometry_def}
}
The maximum integer $m$ for which \eqref{isometry_def} holds (possibly infinite) is called the \textit{maximum degree} of $\phi$, and is denoted $\deg(\phi)$.
The following lemmas demonstrate two useful properties of isometries: composition and extension.

\begin{lemma}[{\cite[Lemma 2.2]{bates-chatterjee17}}] \label{composition}
Let $\phi : A \to \N \times \Z^d$ and $\psi : B \to \N \times \Z^d$ be isometries.
Define $A' \coloneqq \{a \in A : \phi(a) \in B\}$.  Then $\theta: A' \to \N \times \Z^d$ defined by $\theta(u) = \psi(\phi(u))$ is an isometry with $\deg(\theta) \geq \min(\deg(\phi),\deg(\psi))$.
\end{lemma}

\begin{lemma}[{\cite[Lemma 2.3]{bates-chatterjee17}}] \label{extension}
Suppose that $\phi : A \to \N \times \Z^d$ is an isometry of degree $m \geq 3$.
Then $\phi$ can be extended to an isometry $\Phi: A^{(1)} \to \N \times \Z^d$ of degree $m-2$, where
\eq{
A^{(1)} \coloneqq \{v \in \N \times \Z^d : \|u - v\|_1 \leq 1 \text{ for some $u \in A$}\} \supset A.
}
By induction, if $\phi$ has $\deg(\phi) \geq 2k + m$, then $\phi$ can be extended to an isometry $\Phi : A^{(k)} \to \N \times \Z^d$ of degree $m$, where
\eq{
A^{(k)} \coloneqq \{v \in \N \times \Z^d : \|u - v\|_1 \leq k \text{ for some $u \in A$}\} \supset A.
}
\end{lemma}

Given an isometry $\phi$ (which implicitly stands for the pair $(A,\phi)$) and $\alpha > 1$, we define the $\alpha$-distance function according to $\phi$:
\eq{
d_{\alpha,\phi}(f,g) \coloneqq \alpha\sum_{u \in A} |f(u) - g(\phi(u))| + \sum_{u \notin A} f(u)^\alpha + \sum_{u \notin \phi(A)} g(u)^\alpha + 2^{-\deg(\phi)}, \quad f,g \in \s_0.
}
Finally, the pseudometric is obtained by taking the optimal $\alpha$-distance:
\eq{
d_\alpha(f,g) \coloneqq \inf_{\phi\, :\, \deg(\phi) \geq 1} d_{\alpha,\phi}(f,g), \quad 
}
where $\deg(\phi) \geq 1$ means $\phi$ is injective.
The case $\alpha=2$ was considered in \cite{bates-chatterjee17}, and we can easily adapt the proof given there to show $d_\alpha$ satisfies the triangle inequality.
Since $d_\alpha$ is clearly symmetric in $f$ and $g$ (by changing $\phi$ to $\phi^{-1}$), this result verifies that $d_\alpha$ is a pseudometric.

\begin{lemma}
For any $f,g,h \in \s_0$,
\eeq{
d_\alpha(f,h) \leq d_\alpha(f,g) + d_\alpha(g,h). \label{triangle}
}
\end{lemma}

\begin{proof}
Fix $\eps > 0$, and choose isometries $\phi : A \to \N \times \Z^d$ and $\psi : B \to \N \times \Z^d$ such that
\eq{
d_{\alpha,\phi}(f,g) < d_\alpha(f,g) + \eps \quad \text{and} \quad
d_{\alpha,\psi}(g,h) < d_\alpha(g,h) + \eps.
}
Define $\theta : A' \to \N \times \Z^d$ as in Lemma \ref{composition}.
We have
\eeq{
d_{\alpha,\theta}(f,h) = \alpha\sum_{u \in A'} |f(u) - h(\theta(u))| + \sum_{u \notin A'} f(u)^\alpha + \sum_{u \notin \theta(A')} h(u)^\alpha + 2^{-\deg(\theta)}. \label{theta_distance}
}
The first sum above can be bounded as
\eeq{
\alpha\sum_{u \in A'} |f(u) - h(\theta(u))| &\leq \alpha\sum_{u \in A'} \bigl(|f(u) - g(\phi(u))| + |g(\phi(u)) - h(\psi(\phi(u)))|\bigr)  \\
&= \alpha\sum_{u \in A'} |f(u) - g(\phi(u))| + \alpha\sum_{u \in B \cap \phi(A)} |g(u) - h(\psi(u))|. \label{tri_bound_1}
}
Now, the Lipschitz norm of the function $t \mapsto t^\alpha$ on $[0,1]$ is $\alpha$, meaning
\eq{
f(u)^\alpha \leq \alpha|f(u)-g(v)| + g(v)^\alpha \quad \text{for any $u,v \in \N \times \Z^d$}.
}
As a result, the second sum in \eqref{theta_distance} satisfies
\eeq{
\sum_{u \notin A'} f(u)^{\alpha} &= \sum_{u \in A \setminus A'} f(u)^\alpha + \sum_{u \notin A} f(u)^\alpha  \\
&\leq \sum_{u \in A \setminus A'}\bigl( \alpha|f(u) - g(\phi(u))| + g(\phi(u))^\alpha\bigr) + \sum_{u \notin A} f(u)^\alpha  \\
&\leq \alpha\sum_{u \in A \setminus A'} |f(u) - g(\phi(u))| + \sum_{u \notin B} g(u)^\alpha + \sum_{u \notin A} f(u)^\alpha. \label{tri_bound_2}
}
Similarly, the third sum satisfies
\eeq{
\sum_{u \notin \theta(A')} h(u)^\alpha &= \sum_{u \in \psi(B) \setminus \theta(A')} h(u)^\alpha + \sum_{u \notin \psi(B)} h(u)^{\alpha}  \\
&\leq \sum_{u \in B \setminus \phi(A)} \bigl(\alpha|h(\psi(u)) - g(u)| + g(u)^\alpha\bigr) + \sum_{u \notin \psi(B)} h(u)^{\alpha}  \\
&\leq \alpha\sum_{u \in B \setminus \phi(A)} |g(u) - h(\psi(u))| + \sum_{u \notin \phi(A)} g(u)^{\alpha}+ \sum_{u \notin \psi(B)} h(u)^{\alpha}. \label{tri_bound_3}
}
Finally, Lemma \ref{composition} guarantees 
\eeq{
\deg(\theta) &\geq \min(\deg(\phi),\deg(\psi))  \\
\quad \Rightarrow 2^{-\deg(\theta)} &\leq 2^{-\deg(\phi)} + 2^{-\deg(\psi)}. \label{tri_bound_4}
}
Using \eqref{tri_bound_1}--\eqref{tri_bound_4} in \eqref{theta_distance}, we find
\eq{
d_\alpha(f,h) \leq d_{\alpha,\theta}(f,h) \leq d_{\alpha,\phi}(f,g) + d_{\alpha,\psi}(g,h) < d_\alpha(f,g) + d_\alpha(g,h) + 2\eps.
}
As $\eps$ is arbitrary, \eqref{triangle} follows.
\end{proof}

With this pseudometric, a new space is realized by taking the quotient of $\s_0$ with respect to $d_\alpha$:
\eq{
\s \coloneqq \s_0/(d_\alpha = 0).
}
That is, $\s$ is the set of equivalence classes of $\s_0$ under the equivalence relation
\eq{
f \equiv g \quad \Leftrightarrow \quad d_\alpha(f,g) = 0.
}
We call $\s$ the space of \textit{partitioned subprobability measures}, and it naturally inherits the metric $d_\alpha$.
Lemma \ref{equivalence_classes} below shows that for distinct $\alpha,\alpha' > 1$, we have $d_\alpha(f,g) = 0$ if and only if $d_{\alpha'}(f,g) = 0$.
Therefore, we are justified in not decorating the space $\s$ with an $\alpha$ parameter, since $\s_0/(d_\alpha = 0)$ is always the same set.

The quotient map $\iota : \s_0 \to \s$ that sends an element to its equivalence class is Borel measurable with respect to the metric topology, see \cite[Lemma 2.12]{bates-chatterjee17}.
By definition, $d_\alpha(f,g)$ can be evaluated at any representatives for $f$ and $g$.
The lemma below gives a complete description of the equivalence classes; roughly speaking, they are the orbits under translations.
It should be noted that this lemma 
was proved in \cite[Corollary 2.6]{bates-chatterjee17} with $\alpha=2$.
Nevertheless, the proof found there is immediately generalized by replacing the $2$'s with $\alpha$'s.   

\begin{lemma}[{see \cite[Corollary 2.6]{bates-chatterjee17}}] \label{equivalence_classes}
Define the {\normalfont $\N$-support} of $f \in \s_0$ to be the set
\eq{
H_f \coloneqq \{n \in \N : f(n,x) > 0 \text{ for some $x \in \Z^d$}\}.
}
The {\normalfont support number} of $f$ is the cardinality of $H_f$, which is possibly infinite.
For $g \in \s_0$ with $\N$-support $H_g$,
$d_\alpha(f,g) = 0$ if and only if there is a bijection $\sigma : H_f \to H_g$ and vectors $(x_n)_{n \in H_f}$ in $\Z^d$ such that 
\eeq{
g(\sigma(n),x) = f(n,x+x_n) \quad \text{for all $n \in H_f$, $x \in \Z^d$.} \label{better_def}
}
\end{lemma}


The key fact, and indeed the goal of constructing $\s$, is the following result.
It was proved for $\alpha = 2$ in \cite[Theorem 2.9]{bates-chatterjee17}, and once more the proof readily extends to any $\alpha > 1$ by a modification as simple as changing the $2$'s to $\alpha$'s.

\begin{lemma}[{cf.~\cite[Theorem 2.9]{bates-chatterjee17}}] \label{compactness}
$(\s,d_\alpha)$ is a compact metric space.
\end{lemma}

We will generally write $f$ for an element of $\s$, and explicitly indicate $f \in \s_0$ when referring to a representative in $\s_0$.
When $f$ is being evaluated at some $u \in \N \times \Z^d$, a representative has been chosen.
For certain global functionals such as $\|\cdot\|$ defined in \eqref{norm_def}, the choice of representative does not matter, and we can safely write $\|f\|$ without reference to a particular representative.
To see necessary and sufficient conditions for a function to have this property, the reader is referred to \cite[Corollary 2.7]{bates-chatterjee17}.
The result there gives conditions for checking that a function of $\s_0$ is well defined on $\s$ (i.e.~is constant on equivalence classes).

%

Finally, we consider the space of (Borel) probability measures on $\s$, denoted $\p(\s)$, together with the Wasserstein metric (see \cite[Definition 6.4]{villani09}),
\eq{
\w_\alpha(\mu,\nu) \coloneqq \inf_{\pi \in \Pi(\mu,\nu)} \int_{\s\times\s} d_\alpha(f,g)\ \pi(\dd f,\dd g).
}
Here $\Pi(\mu,\nu)$ denotes the set of probability measures on $\s \times \s$ having $\mu$ and $\nu$ as marginals.
Lemma \ref{compactness} implies $(\p(\s),\w_\alpha)$ is also a compact metric space.

\begin{lemma}[{see \cite[Remark 6.19]{villani09}}]
$(\p(\s),\w_\alpha)$ is a compact metric space.
\end{lemma}

It is a standard fact (for instance, see \cite[Theorem 6.9]{villani09}) that Wasserstein distance metrizes the topology of weak convergence.
In the \textit{compact} setting, weak convergence is equivalent to convergence of continuous test functions.

\begin{lemma}[{Portmanteau, see \cite[Theorem 2.1]{billingsley99} and \cite[Theorem 1.3.4]{vandervaart-wellner96}}] \label{portmanteau}
Given a function $L: \s \to \R$, define the map $\l : \p(\s) \to \R$ by
\eq{
\l(\mu) \coloneqq \int_\s L(f)\ \mu(\dd f).
}
If $L$ is (lower/upper semi-)continuous, then $\l$ is (lower/upper semi-)continuous.
\end{lemma}

\subsection{Equivalence of generalized metrics} \label{generalized_topology}
We have introduced a family of metrics $(d_\alpha)_{\alpha>1}$ on $\s$, where the flexibility of choosing $\alpha$ sufficiently close to $1$ will allow us to make more effective use of the abstract methods in \cite{bates-chatterjee17}.
Namely, the only assumption we need is \eqref{mgf_assumption}.
It is important, however, that each metric induces the same topology.
The next proposition verifies this fact.
In particular, any functional on $\s$ that was proved in \cite{bates-chatterjee17} to be continuous with respect to $d_2$ remains continuous under $d_\alpha$, $\alpha > 1$.

\begin{prop} \label{metrics_equivalent}
For any $\alpha,\alpha' > 1$, $f \in \s$, and sequence $(f_n)_{n\geq1}$ in $\s$, we have $d_{\alpha}(f,f_n) \to 0$ as $n \to \infty$ if and only if $d_{\alpha'}(f,f_n) \to 0$.
\end{prop}

\begin{proof}
Since $\alpha$ and $\alpha'$ are interchangeable in the claim, it suffices prove the ``only if" direction.
That is, we assume $d_\alpha(f,f_n) \to 0$ as $n \to \infty$.
Fix representatives $f,f_n \in \s_0$.
Given $\eps > 0$, set
\eq{
\delta \coloneqq \min\Big\{\Big(\frac{\eps}{4}\Big)^{\alpha/(\alpha'-1)},\frac{\alpha}{\alpha'}\Big(\frac{\eps}{4}\Big),\frac{\eps}{4}\Big\}.
}
Then choose $N$ sufficiently large that $d_\alpha(f,f_n) < \delta$ for all $n \geq N$.
In particular, for any such $n$, there is an isometry $\phi_n : A_n \to \N \times \Z^d$ satisfying
$d_{\alpha,\phi_n}(f,f_n) < \delta$.
In particular,
\eq{
\alpha'\sum_{u \in A_n} |f(u)-f_n(\phi_n(u))| \leq \frac{\alpha'}{\alpha} d_{\alpha,\phi_n}(f,f_n) < \frac{\alpha'}{\alpha}\delta < \frac{\eps}{4}.
}
Also,
\eq{
\sum_{u \notin A_n} f(u)^{\alpha'} \leq \max_{u \notin A_n} f(u)^{\alpha'-1} \sum_{u \notin A_n} f(u)
\leq \Big(\max_{u \notin A_n} f(u)\Big)^{\alpha'-1} \leq d_{\alpha,\phi_n}(f,f_n)^{(\alpha'-1)/\alpha} 
< \delta^{(\alpha'-1)/\alpha} < \frac{\eps}{4},
}
and similarly
\eq{
\sum_{u \notin \phi_n(A_n)} f_n(u)^{\alpha'} < \delta^{(\alpha'-1)/\alpha} < \frac{\eps}{4}.
}
Finally,
\eq{
2^{-\deg(\phi_n)} \leq d_{\alpha,\phi_n}(f,f_n) < \delta < \frac{\eps}{4}.
}
These four inequalities together show
\eq{
d_{\alpha'}(f,f_n) \leq d_{\alpha',\phi_n}(f,f_n) < \eps \quad \text{for all $n \geq N$.}
}
As $\eps > 0$ is arbitrary, it follows that $d_{\alpha'}(f,f_n) \to 0$.

\end{proof}

\subsection{Generalized update map} \label{update_map}
Throughout the remainder of the manuscript, we fix $\beta \in (0,\beta_{\max}) $ according to \eqref{mgf_assumption}, and we also fix some $\alpha > 1$ such that $\alpha\beta < \beta_{\max}$.
We then restrict our attention to $\s$ equipped with the metric $d_\alpha$, and $\p(\s)$ with $\w_\alpha$.
Proposition \ref{metrics_equivalent} tells us that the topology on $\s$ does not depend on $\alpha$, although the same is not true for the topology on $\p(\s)$ induced by $\w_\alpha$.
Indeed, there can exist functions $\vphi : \s \to \R$ which are Lipschitz-1 with respect to some $d_\alpha$ but not Lipschitz at all with respect to some other $d_{\alpha'}$.

We write $f_n$ to denote the (random) endpoint distribution under the polymer measure $\rho_n$, belonging to either $\ell^1(\Z^d)$ or $\s$ depending on context.
Notice that we have the recursion
\eq{
f_n(x) = \frac{Z_n(x)}{Z_n}
&= \frac{1}{Z_n}\sum_{x_1,\cdots,x_{n-1}} \exp\bigg(\beta\sum_{i = 1}^{n-1} \eta(i,\omega_i)\bigg)e^{\beta\eta(n,x)}\bigg(\prod_{i=1}^{n-1} P(x_{i-1},x_i)\bigg)P(x_{n-1},x) \\
&= \frac{Z_{n-1}}{Z_n} \sum_{x_{n-1}\in\Z^d} \frac{Z_{n-1}(x_{n-1})}{Z_{n-1}}e^{\beta\eta(n,x)}P(x_{n-1},x) \\
&= \frac{Z_{n-1}}{Z_n} \sum_{x_{n-1} \in \Z^d} f_{n-1}(x_{n-1})e^{\beta\eta(n,x)}P(x_{n-1},x).
} 
which implies
\eeq{
\frac{Z_n}{Z_{n-1}} = \frac{Z_n}{Z_{n-1}} \sum_{z \in \Z^d} f_n(z) 
= \sum_{z \in \Z^d} \sum_{y \in \Z^d} f_{n-1}(y)e^{\beta\eta(n,z)}P(y,z). \label{ratio_equality}
}
Alternatively, we have
\eq{
f_n(x) = \frac{\sum_{y} f_{n-1}(y)e^{\beta\eta(n,x)}P(y,x)}{\sum_z \sum_y f_{n-1}(y)e^{\beta\eta(n,z)}P(y,z)}.
}
This identity shows how $f_0 \mapsto f_1 \mapsto \cdots $ forms a Markov chain when embedded into $\s$.
Namely, we identify $f_{n-1}$ with its equivalence class in $\s$ so that a representative takes values on $\N \times \Z^d$ instead of $\Z^d$. (In this case, the support number is just 1.)
Then the law of $f_n \in \s$ given $f = f_{n-1}$ is the law of the random variable $F \in \s$ defined by
\eeq{ \label{F_def_0}
F(u) = \frac{\sum_{v \in \N \times \Z^d} f(v)e^{\beta\eta_u}P(v,u)}{\sum_{w\in \N \times \Z^d}\sum_{v\in \N \times \Z^d} f(v)e^{\beta\eta_w}P(v,w)},
}
where $(\eta_w)_{w \in \N \times \Z^d}$ is an i.i.d.~collection of random variables having the same law as $\eta$, and
\eq{
P(v,w) = \begin{cases}
P(y,z) &\text{if $v = (n,y)$, $w = (n,z)$} \\
0 &\text{otherwise}.
\end{cases}
}
To simplify notation, we write $v \sim w$ in the first case (i.e.~$v$ and $w$ have the same first coordinate) and $v \nsim w$ otherwise. 

\begin{remark}
Although the indexing of $\eta_w$ by $w \in \N \times \Z^d$ might appear to reflect a notion of time, we are not using $\N$ to consider time.
Rather, in order to compactify the space of measures on $\Z^d$, we needed to pass to subprobability measures on $\N \times \Z^d$.
To avoid confusion, we will never write $\N$ to index time.
Following this rule, we will write $\eta_w$ whenever we wish to think of a random environment on $\N \times \Z^d$, always at a \textit{fixed} time.
When considering the original random environment defining the polymer measures, we will follow the standard $\eta(i,x)$ notation.
In either case, we will continue to use boldface $\vc\eta$ when referring to the entire collection of environment random variables.
\end{remark}

Generalizing \eqref{F_def_0} to $f \in \s$ that may have $\|f\| < 1$, we define $Tf \in \p(\s)$ to be the law of $F \in \s$ defined by
\eeq{
F(u) = \frac{\sum_{v\sim u} f(v)e^{\beta\eta_u}P(v,u)}{\sum_{w\in\N\times\Z^d}\sum_{v\sim w} f(v)e^{\beta\eta_w}P(v,w)+(1-\|f\|)e^{\lambda(\beta)}}. \label{F_def}
}
Notice that the expectation (with respect to $\vc \eta$) of the numerator is $e^{\lambda(\beta)}\sum_{v \sim u} f(v)P(v,u)$, while the expectation of the denominator is $e^{\lambda(\beta)}$.
Therefore, these quantities are almost surely finite, and so $F$ is well-defined.
In order for $Tf$ to be well-defined, we must check the following:
\begin{itemize}
\item[(i)] Given any $f \in \s_0$, the map $\R^{\N\times\Z^d} \to \s$ given by $\vc\eta \mapsto F$ is Borel measurable, where $\R^{\N\times\Z^d}$ is equipped with the product topology and product measure $(\mathfrak{L}_\eta)^{\otimes \N\times\Z^d}$, and $\mathfrak{L}_\eta$ is the law of $\eta$.
\item[(ii)] The law of $F$ does not depend on the representative $f \in \s_0$.
\end{itemize}
Claim (i) is immediate, since $\vc \eta \mapsto F$ is clearly a measurable map from $\R^{\N\times\Z^d}$ to $\s_0$.
After all, it is simply the quotient of sums of measurable functions.
And then $F \to \iota(F)$ from $\s_0$ to $\s$ is measurable by \cite[Lemma 2.12]{bates-chatterjee17}.
Claim (ii) is given by the following lemma.

\begin{lemma}[{cf.~\cite[Proposition 3.1]{bates-chatterjee17}}] \label{same_law}
Suppose $f,g \in \s_0$ satisfy $d_\alpha(f,g) = 0$.
Define $F$ as in $\eqref{F_def}$, and similarly define
\eeq{
G(u) \coloneqq \frac{\sum_{v\sim u} g(v) e^{\beta \zeta_u}P(v,u)}{\sum_{w\in\N\times\Z^d} \sum_{v\sim w} g(v)e^{\beta \zeta_w}P(v,w) + (1 - \|g\|)e^{\lambda(\beta)}},\quad u \in \N \times \Z^d, \label{G_def}
}
where the $\zeta_w$ are i.i.d., each having law $\mathfrak{L}_\eta$.
Then when these functions are mapped into $\s$ by $\iota$, the law of $F$ is equal to the law of $G$.
\end{lemma}

\begin{proof}
To show that $F$ and $G$ have the same law, it suffices to exhibit a coupling of the environments $\vc\eta$ and $\vc\zeta$ such that $F = G$ in $\s$.
So we let $H_f$ and $H_g$ denote the $\N$-supports of $f$ and $g$, respectively, and take
$\sigma : H_f \to H_g$ and $(x_n)_{n \in H_f}$ as in Lemma \ref{equivalence_classes}.
Define the translation(s) $\psi : H_f \times \Z^d \to H_g \times \Z^d$ by $\psi(n,x) \coloneqq (\sigma(n),x-x_n)$, so that \eqref{better_def} now reads as
\eeq{
f(v) = g(\psi(v)) \quad \text{for all $v \in H_f \times \Z^d$}. \label{f_equals_g}
}

Next we couple the environments.
Let $\zeta_u$ be equal to $\eta_{\psi^{-1}(u)}$ whenever $u \in H_g \times \Z^d$.
Otherwise, we may take $\zeta_u$ to be an independent copy of $\eta_u$.
Now, for any $u = (n,x)$ and $v = (n,y)$ with $n \in H_f$,
\eeq{
\psi(u) - \psi(v) = u - v \quad \Rightarrow \quad P(v,u) &= P(\omega_1 = u - v) \\ &= P(\omega_1 = \psi(u)-\psi(v)) = P(\psi(v),\psi(u)). \label{transitions_equal}
}
Therefore, for any $u \in H_f \times \Z^d$, we have
\eeq{ \label{numerator_same}
\sum_{v \sim u} f(v)e^{\beta \eta_u}P(v,u)
&\stackrel{\mbox{\scriptsize\eqref{f_equals_g},\eqref{transitions_equal}}}{=} \sum_{v \sim u} g(\psi(v))e^{\beta \eta_u}P(\psi(v),\psi(u)) \\
&\stackrel{\phantom{\eqref{f_equals_g},\eqref{transitions_equal}}}{=} \sum_{v \sim \psi(u)} g(v)e^{\beta \zeta_{\psi(u)}}P(v,\psi(u)).
}
On the other hand, for any $u \in (\N \setminus H_f) \times \Z^d$ we have $f(v) = 0$ for every $v \sim u$. 
Similarly, $g(v) = 0$ whenever $u \in (\N \setminus H_g) \times \Z^d$ and $v \sim u$.
Consequently, \eqref{numerator_same} trivially implies
\eeq{
\sum_{w \in \N \times \Z^d} \sum_{v \sim w} f(v) e^{\beta \eta_w}P(v,w)
= \sum_{w \in \N \times \Z^d} \sum_{v \sim w} g(v) e^{\beta \zeta_w}P(v,w). \label{denominator_same}
}
Together, \eqref{numerator_same}, \eqref{denominator_same}, and the fact that $\|f\| = \|g\|$ (see discussion following Lemma \ref{compactness}) show $F(u) = G(\psi(u))$ for all $u \in H_f \times \Z^d$.
Hence
\eq{
\sum_{u \in H_f \times \Z^d} |F(u) - G(\psi(u))| + \sum_{u \notin H_f \times \Z^d} F(u)^\alpha + \sum_{u \notin H_g \times \Z^d} G(u)^\alpha = 0.
}
For any $\eps > 0$, we can find a finite subset $A \subset H_f \times \Z^d$ such that
\eq{
\sum_{u \notin A} F(u)^\alpha + \sum_{u \notin \psi(A)} G(u)^\alpha < \eps.
}
With $\phi \coloneqq \psi\rvert_{A}$ (so $\deg(\phi) = \infty$), we thus have $d_\alpha(F,G) \leq d_{\alpha,\phi}(F,G) < \eps$.
Letting $\eps$ tend to 0 gives the desired result.
\end{proof}

We have now verified that the map $\s \to \p(\s)$ given by $f \mapsto Tf$ is well-defined.
It remains to be seen that the map is measurable, although this fact will be implied by the continuity proved in the next section.
Given measurability, we can naturally lift the update map to a map on measures.
For $\mu \in \p(\s)$, define the mixture
\eeq{
\t\mu(\dd g) \coloneqq \int Tf(\dd g)\, \mu(\dd f), \label{t_on_measures}
}
which means
\eq{
\t\mu(\a) = \int Tf(\a)\ \mu(\dd f), \quad \text{Borel $\a \subset \s$}.
}
More generally,
\eeq{
\int \vphi(g)\ \t\mu(\dd g) = \int_\s \int_\s \vphi(g)\ Tf(\dd g)\, \mu(\dd f) \label{T_fubini}
}
for all measurable functions $\vphi : \s \to \R$ that either are nonnegative or satisfy
\eq{
\int_\s \int_\s |\vphi(g)|\ Tf(\dd g)\,\mu(\dd f) < \infty. 
}
In this notation, we have a map $\p(\s) \to \p(\s)$ given by $\mu \mapsto \t\mu$.
We can recover the map $T$ by restricting to Dirac measures;
that is, $Tf = \t\delta_f$, where $\delta_f \in \p(\s)$ is the unit point mass at $f$.
For our purposes here, it suffices to know that $\mu \mapsto \t\mu$ is continuous, by an argument which requires only the continuity of $f \mapsto Tf$ (see \cite[Appendix B.1]{bates-chatterjee17}).

\subsection{Continuity of update map} \label{continuity}
The goal of this section is to prove the following result.

\begin{prop}[{cf.~\cite[Proposition 3.2]{bates-chatterjee17}}] \label{continuous1}
For any $\eps > 0$, there exists $\delta > 0$ such that for $f,g \in \s$,
\eq{
d_\alpha(f,g) < \delta \quad \Rightarrow \quad \w_\alpha(Tf,Tg) < \eps.
}
\end{prop}

The proof will proceed in a similar manner as in the nearest-neighbor random walk case, although modification is necessary to account for the fact that now the set of ``neighbors" may be arbitrarily large, even all of $\Z^d$.
In preparation, we record the following results.

\begin{lemma}[{\cite[Lemma 3.3]{bates-chatterjee17}}] \label{amgm}
Let $X_1,X_2,\dots$ be i.i.d.~copies of a positive random variable $X$.
If $c_1,c_2,\dots$ are nonnegative constants 
 satisfying $C \leq \sum_{i = 1}^\infty c_i < \infty $ for a positive constant $C$, then
\eq{
\EE\Bigg[\bigg(\sum_{i=1}^\infty c_i X_i\bigg)^{-p}\Bigg] \leq C^{-p}\, \EE(X^{-p}) \quad \text{for any $p > 0$.}
}
\end{lemma}

\begin{lemma} \label{LLN}
Let $X_1,X_2,\dots$  be i.i.d.~copies of an integrable, centered random variable $X$.
For any $t > 0$, there exists $b > 0$ such that whenever $c_1,c_2,\dots$ are constants satisfying $|c_i| \leq b$ for all $i$ and $\sum_{i = 1}^\infty|c_i| \leq 1$, then
\eeq{
\EE\bigg|\sum_{i = 1}^\infty c_i X_i\bigg| \leq t. \label{LLN_to_show}
}
\end{lemma}

\begin{proof}
Fix $t > 0$ and choose $L$ sufficiently large that if $\bar X \coloneqq X\one_{\{|X|\leq L\}}$, then
\eeq{ \label{truncation_1}
\EE|X-\bar X| \leq \frac{t}{4} \wedge 1.
}
In particular,
\eeq{ \label{truncation_2}
|\EE(\bar X)| = |\EE(\bar X - X)| \leq \EE|\bar X - X| \leq \frac{t}{4} \wedge 1.
}
Given $t$ and $L$, let $b>0$ be sufficiently small that
\eeq{
e^{-t^2/(128(L+1)^2b)} \leq \frac{t}{8(L+1)}. \label{exponential_bound}
}
As in the hypothesis, assume $|c_i| \leq b$ for all $i$, and $\sum_{i=1}^\infty |c_i| \leq 1$.
Consider that
\eeq{ \label{first_triangle}
\EE\bigg|\sum_{i = 1}^\infty c_i X_i\bigg|
= \EE\bigg|\sum_{i = 1}^\infty c_i\bar X_i + c_i(X-\bar X_i)\bigg|
&\stackrel{\phantom{\eqref{truncation_1}}}{\leq} \EE\bigg|\sum_{i = 1}^\infty c_i\bar X_i\bigg|
+ \sum_{i=1}^\infty |c_i|\EE|X-\bar X| \\
&\stackrel{\mbox{\scriptsize\eqref{truncation_1}}}{\leq}  \EE\bigg|\sum_{i = 1}^\infty c_i\bar X_i\bigg|+ \frac{t}{4}.
}
In order to apply martingale inequalities, we recenter the remaining sum:
\eeq{ \label{second_triangle}
 \EE\bigg|\sum_{i = 1}^\infty c_i\bar X_i\bigg|
 \leq  \EE\bigg|\sum_{i = 1}^\infty c_i(\bar X_i-\EE(\bar X))\bigg| + \sum_{i=1}^\infty |c_i\EE(\bar X)| \stackrel{\mbox{\scriptsize\eqref{truncation_2}}}{\leq} \EE\bigg|\sum_{i = 1}^\infty c_i(\bar X_i-\EE(\bar X))\bigg| + \frac{t}{4}.
}
For each $i$, the random variable $c_i(\bar X_i - \EE(\bar X))$ has mean $0$ and takes values between $-|c_i|(L+1)$ and $|c_i|(L+1)$.
Therefore, by the Azuma--Hoeffding inequality \cite[Theorem 2.8]{boucheron-lugosi-massart13},
\eq{
\PP\bigg(\bigg|\sum_{i = 1}^\infty c_i(\bar X_i-\EE(\bar X))\bigg| > \frac{t}{4}\bigg) &\leq 2\exp\Big(\frac{-(t/4)^2}{2\sum_{i = 1}^\infty (2|c_i|(L+1))^2}\Big) \\
&= 2\exp\Big(\frac{-t^2}{128(L+1)^2\sum_{i=1}^\infty |c_i|^2}\Big) \\
&\leq 2\exp\Big(\frac{-t^2}{128(L+1)^2b}\Big) \stackrel{\mbox{\scriptsize\eqref{exponential_bound}}}{\leq} \frac{t}{4(L+1)}.
}
Since $|\sum_{i = 1}^\infty c_i(\bar X_i - \EE(\bar X))| \leq L+1$, we conclude
\eeq{ \label{third_triangle}
\EE\bigg|\sum_{i = 1}^\infty c_i(\bar X_i-\EE(\bar X))\bigg| \leq \frac{t}{2}.
}
Together, the inequalities \eqref{first_triangle}--\eqref{third_triangle} give \eqref{LLN_to_show}.
\end{proof}

\begin{proof}[Proof of Proposition \ref{continuous1}]
Since $(\s,d_\alpha)$ is a compact metric space, uniform continuity of $f \mapsto Tf$ will be implied by continuity.
So it suffices to prove continuity at a fixed $f \in \s$.

Let $\eps > 0$ be given.
To prove continuity at $f \in \s$, we need to exhibit $\delta > 0$ such that if $g \in \s$ satisfies $d_\alpha(f,g) < \delta$, then there exist representatives $f,g \in \s_0$ and a coupling of environments $(\vc\eta,\vc\zeta)$ such that under this coupling the following inequality holds:
\eq{
\EE(d_\alpha(F,G)) < \eps.
}
Here $F,G\in\s$ are given by \eqref{F_def} and \eqref{G_def}.
We shall see that it suffices to choose $\delta$ satisfying conditions \eqref{delta_condition_k}--\eqref{delta_condition_min} below.

Let $q \coloneqq \max_{x \in \Z^d} P(0,x)$, which is strictly less than $1$ by \eqref{walk_assumption}.
Next choose $t > 0$ sufficiently small to satisfy one of the following conditions:
\begin{subequations}
\label{t_conditions}
\begin{align}
\frac{3t}{(1-\|f\|)e^{\lambda(\beta)}} &< \frac{\eps}{8\alpha}\phantom{\frac{\eps}{16\alpha}} \qquad \text{if $\|f\| < 1$,} \label{t_condition_1} \\
\frac{85t}{36}e^{\lambda(-\beta)} &< \frac{\eps}{16\alpha}\phantom{\frac{\eps}{8\alpha}} \qquad \text{if $\|f\| = 1$.} \label{t_condition_2}
\end{align}
\end{subequations}
Given $t$, we can take $b > 0$ as in Lemma \ref{LLN}, and then choose $\kappa > 0$ to satisfy
\begin{subequations}
\label{kappa_conditions}
\begin{align}
\kappa &\leq \frac{1}{7} \label{kappa_condition_1} \\
14\kappa e^{\lambda(\beta)+\lambda(-\beta)} &< \frac{\eps}{8\alpha} \label{kappa_condition_4} \\
7\kappa + (2\kappa)^{1/\alpha} &\leq b \label{kappa_condition_2} \\
7\kappa(1-q)^{-1}e^{\lambda(-\beta)}\EE|e^{\beta\eta}-e^{\lambda(\beta)}| &<\frac{\eps}{16\alpha}\label{kappa_condition_3} \\
3 \cdot 2^{\alpha} e^{\lambda(\alpha\beta)+\lambda(-\alpha\beta)}\kappa &< \frac{\eps}{4}. \label{kappa_condition_5}
\end{align}
\end{subequations}
Then fix a representative $f \in \s_0$, and choose $A \subset \N \times \Z^d$ finite but sufficiently large that
\eeq{
\sum_{u \notin A} f(u) < \kappa. \label{A_condition}
}
By possibly omitting some elements, we may assume $f$ is strictly positive on $A$.
Next, let $k$ be a positive integer sufficiently large that
\begin{subequations}
\label{k_conditions}
\begin{align}
2^{-k} &< \frac{\eps}{4} \label{k_condition_1} \\
P(\|\omega_1\|_1 > k) &< \kappa. \label{k_condition_2}
\end{align}
\end{subequations}

Now choose $\delta > 0$ such that
\begin{subequations}
\label{delta_conditions}
\begin{align}
  \delta &< 2^{-3k}. \label{delta_condition_k}
\intertext{If $A$ is nonempty, we will also demand that}
(2d)^k|A|\delta^{1/\alpha} &< \kappa. \label{delta_condition_A}
\intertext{Otherwise, we relax this assumption to}
  \delta^{1/\alpha} &< \kappa. \label{delta_condition_noA}
\intertext{Finally, we assume}
\delta^{1/\alpha} &< \inf_{u \in A} f(u). \label{delta_condition_min}
\end{align}
\end{subequations}
We claim this $\delta > 0$ is sufficient for $\eps > 0$ in the sense described above.

Assume $g \in \s$ satisfies $d_\alpha(f,g) < \delta$.
Then given any representative $g \in \s_0$, there exists an isometry $\psi : C \to \N \times \Z^d$ such that
\eeq{
d_{\alpha,\psi}(f,g) = \alpha\sum_{u \in C} |f(u)-g(\psi(u))| + \sum_{u\notin C}f(u)^\alpha + \sum_{u\notin \psi(C)}g(u)^\alpha + 2^{-\deg(\psi)} < \delta. \label{psi_condition}
}
By \eqref{delta_condition_k}, it follows that $\deg(\psi) > 3k$. 
In particular, upon defining $\phi \coloneqq \psi|_A$, we have $\deg(\phi) \geq \deg(\psi) > 3k$.
Therefore, Lemma \ref{extension} guarantees that $\phi$ can be extended to an isometry $\Phi : A^{(k)} \to \N \times \Z^d$ with $\deg(\Phi) > k$, where
\eq{
A^{(k)} = \{v \in \N \times \Z^d: \|u - v\|_1 \leq k \text{ for some $u \in A$}\}.
}
Alternatively, we can express this set as
\eq{
A^{(k)} = \bigcup_{u \in A} \n(u,k), \qquad \n(u,k) \coloneqq \{v \in \N \times \Z^d: \|u - v\|_1 \leq k\}.
}
The inequality $\deg(\Phi) > k$ implies that for any $B \subset A$,
\eeq{ \label{degree_consequence}
\Phi(B^{(k)}) = \Phi(B)^{(k)} \coloneqq \bigcup_{u \in B} \n(\Phi(u),k).
}
Indeed,
\eq{
u \in B^{(k)} \quad \Leftrightarrow \quad \exists\, v \in B \cap \n(u,k) \quad
\Leftrightarrow \quad \exists\, v \in B, \Phi(v) \in \n(\Phi(u),k) \quad
\Leftrightarrow \quad \Phi(u) \in \Phi(B)^{(k)}.
}
Furthermore, \eqref{psi_condition} and \eqref{delta_condition_min} together force $A \subset C$, since
\eq{
u \in A \quad \Rightarrow \quad f(u) \geq \inf_{u \in A} f(u) > \delta^{1/\alpha}
\quad \Rightarrow \quad f(u)^\alpha > \delta > d_\psi(f,g) \quad \Rightarrow \quad u \in C.
}

Given a random environment $\vc\eta$, we couple to it $\vc\zeta$ in the following way.
If $u \in \Phi(A^{(k)})$, then set $\zeta_u = \eta_{\Phi^{-1}(u)}$.
Otherwise, let $\zeta_u$ be an independent copy of $\eta_u$.
Note that $F$ and $G$ are now distributed as $Tf$ and $Tg$, respectively, when mapped to $\s$.
On the other hand, $\Phi$ is deterministic, and so no measurability issues arise in the bound
\eq{
\EE(d_\alpha(F,G)) \leq \EE(d_{\alpha,\Phi}(F,G)),
}
where by definition
\eeq{ \label{d_Phi_def}
d_{\alpha,\Phi}(F,G) = \delta \sum_{u \in A^{(k)}} |F(u)-G(\Phi(u))| + \sum_{u \notin A^{(k)}} F(u)^\alpha + \sum_{u \notin \Phi(A^{(k)})} G(u)^\alpha + 2^{-\deg(\Phi)}.
}
It thus suffices to show $\EE(d_{\alpha,\Phi}(F,G)) < \eps$.
To simplify notation, we will write
\eq{
\wt{f}(u) &= \sum_{v \sim u} f(v)e^{\beta \eta_u}P(v,u) & \wt{F} &= \sum_{w \in \N \times \Z^d} \wt{f}(w) + (1 - \|f\|)e^{\lambda(\beta)} \\
\wt{g}(u) &= \sum_{v \sim u} g(v)e^{\beta \zeta_u}P(v,u) & \wt{G} &= \sum_{w \in \N \times \Z^d} \wt{g}(w) + (1 - \|g\|)e^{\lambda(\beta)}
}
so that $F(u) = \wt{f}(u)/\wt{F}$ and $G(u) = \wt{g}(u)/\wt{G}$.

Consider the first of the four terms on the right-hand side of \eqref{d_Phi_def}.
Observe that for $u \in A^{(k)}$,
\eeq{
|F(u) - G(\Phi(u))| = \bigg|\frac{\wt{f}(u)}{\wt{F}} - \frac{\wt{g}(\Phi(u))}{\wt{G}}\bigg| 
&\leq \bigg|\frac{\wt{f}(u)}{\wt{F}} - \frac{\wt{g}(\Phi(u))}{\wt{F}}\bigg| + \bigg|\frac{\wt{g}(\Phi(u))}{\wt{F}}-\frac{\wt{g}(\Phi(u))}{\wt{G}}\bigg|  \\
&= \frac{|\wt{f}(u) - \wt{g}(\Phi(u))|}{\wt{F}} + \frac{\wt{g}(\Phi(u))}{\wt{G}}\cdot\bigg|\frac{\wt{G}}{\wt{F}} - 1\bigg|. \label{two_parts}
}
Summing over $A^{(k)}$ and taking expectation, we obtain the following from the first term in the last line of \eqref{two_parts}:
\eeq{ \label{first_part}
\EE\Bigg[\sum_{u \in A^{(k)}} \frac{|\wt{f}(u) - \wt{g}(\Phi(u))|}{\wt{F}}\Bigg]
&= \sum_{u \in A^{(k)}} \EE\bigg[\frac{|\wt{f}(u) - \wt{g}(\Phi(u))|}{\wt{F}}\bigg].
}
Meanwhile, the second term in \eqref{two_parts} gives
\eeq{ \label{second_part}
\EE\Bigg[\sum_{u \in A^{(k)}} \frac{\wt{g}(\Phi(u))}{\wt{G}}\cdot\bigg|\frac{\wt{G}}{\wt{F}}-1\bigg|\Bigg]
\leq \EE\, \bigg|\frac{\wt{G}}{\wt{F}}-1\bigg|
= \EE\bigg[ \frac{|\wt{G}-\wt{F}|}{\wt{F}}\bigg].
}
These preliminary steps suggest two quantities we should control from above, namely the right-hand sides of \eqref{first_part} and \eqref{second_part}.
Consideration of the second and third terms on the right-hand side of \eqref{d_Phi_def} suggests four more quantities,
since the Harris--FKG inequality yields
\eeq{ \label{to_bound_2}
\EE\bigg[\sum_{u \notin A^{(k)}} F(u)^\alpha\bigg]
\leq \EE(\wt F^{-\alpha})\EE\bigg[\sum_{u \notin A^{(k)}} \wt f(u)^\alpha\bigg],
}
and similarly
\eeq{ \label{to_bound_3}
\EE\bigg[\sum_{u \notin \Phi(A^{(k)})} G(u)^\alpha\bigg]
\leq \EE(\wt G^{-\alpha})\EE\bigg[\sum_{u \notin \Phi(A^{(k)})} \wt g(u)^\alpha\bigg].
}
Therefore, we should seek an upper bound for $\EE(\wt F^{-\alpha})$ and $\EE(\wt G^{-\alpha})$, as well as $\EE\big[\sum_{u \notin A^{(k)}} \wt f(u)^\alpha\big]$ and $\EE\big[\sum_{u \notin \Phi(A^{(k)})} \wt g(u)^\alpha\big]$.
For clarity of presentation, we divide our task into the next four subsections.

\subsubsection{Upper bound for $\EE(\wt{F}^{-\alpha})$ and $\EE(\wt{G}^{-\alpha})$}
If $\|f\| < 1$, then
\eeq{ \label{Fcase_1}
\wt F \geq (1-\|f\|)e^{\lambda(\beta)}  \quad \Rightarrow \quad
\EE(\wt F^{-\alpha}) &\leq (1-\|f\|)^{-\alpha}e^{-\alpha\lambda(\beta)} \leq (1-\|f\|)^{-\alpha}e^{\lambda(-\alpha\beta)}.
}
On the other hand,
\eq{
\sum_{u \in \N \times \Z^d} \sum_{v \sim u} f(v)P(v,u) = \sum_{v \in \N \times \Z^d} \sum_{u \sim v} f(v)P(v,u) = \sum_{v \in \N \times \Z^d} f(v) = \|f\|,
}
and so if $\|f\| > 0$, then Lemma \ref{amgm} gives
\eeq{
\EE\, \wt{F}^{-\alpha}
\leq \EE\Bigg[\bigg(\sum_{u \in \N \times \Z^d} \sum_{v \sim u} f(v)e^{\beta \eta_u}P(v,u)\bigg)^{-\alpha}\Bigg]
\leq \|f\|^{-\alpha} \cdot \EE(e^{-\alpha\beta \eta}) = \|f\|^{-\alpha} e^{\lambda(-\alpha\beta)}. \label{Fcase_2}
}
Viewing \eqref{Fcase_1} and \eqref{Fcase_2} together, we have unconditionally that
\eeq{
\EE(\wt{F}^{-\alpha}) \leq \Big[\max_{t \in [0,1]} \min((1-t)^{-\alpha},t^{-\alpha})\Big]e^{\lambda(-\alpha\beta)}
= 2^\alpha e^{\lambda(-\alpha\beta)}. \label{Fbound}
}
The exact same argument shows
\eeq{
\EE(\wt G^{-\alpha}) \leq 2^\alpha e^{\lambda(-\alpha\beta)}. \label{Gbound}
}

\subsubsection{Upper bound for $\sum_{u \in A^{(k)}} \EE(|\wt{f}(u) - \wt{g}(\Phi(u))|/\wt F)$}

Observe that for $u \in A^{(k)}$,
\eq{
|\wt f(u) - \wt g(\Phi(u))|
&= \bigg| \sum_{v \sim u} f(v)e^{\beta \eta_u}P(v,u) - \sum_{v \sim \Phi(u)} g(v)e^{\beta \zeta_{\Phi(u)}}P(v,\Phi(u))\bigg| \\
&= e^{\beta\eta_u}\bigg|\sum_{v \sim u} f(v)P(v,u) - \sum_{v \sim \Phi(u)} g(v)P(v,\Phi(u))\bigg|,
}
meaning $|\wt f(u) - \wt g(\Phi(u))|$ is a non-decreasing function of $\eta_u$ and independent of $\eta_w$ for $w \neq u$.
Since $\wt F^{-1}$ is a non-increasing function of all $\eta_w$, the Harris--FKG inequality yields
\eeq{ \label{initial_fkg}
\sum_{u \in A^{(k)}} \EE\bigg[\frac{|\wt{f}(u) - \wt{g}(\Phi(u))|}{\wt F}\bigg]
\leq \EE(\wt F^{-1}) \sum_{u \in A^{(k)}} \EE|\wt f(u) - \wt g(\Phi(u))|,
}
where
\eeq{  \label{fixed_u}
 \sum_{u \in A^{(k)}}\EE|\wt f(u) - \wt g(\Phi(u))|
&\leq e^{\lambda(\beta)} \sum_{u \in A^{(k)}}\bigg(\bigg|\sum_{v \in \n(u,k)} f(v)P(v,u) - \sum_{v \in \n(\Phi(u),k)} g(v)P(v,\Phi(u))\bigg| \\
&\phantom{=e^{\beta\eta_u}\bigg[}+\bigg|\sum_{\substack{v \sim u \\ v \notin \n(u,k)}} f(v)P(v,u) - \sum_{\substack{v \sim \Phi(u) \\ v \notin \n(\Phi(u),k)}} g(v)P(v,\Phi(u))\bigg|\bigg).
}
Of the two absolute values above, the second is easier to control.
Indeed,
\eeq{ \label{easier_one}
&\sum_{u \in A^{(k)}} \bigg|\sum_{\substack{v \sim u \\ v \notin \n(u,k)}} f(v)P(v,u) - \sum_{\substack{v \sim \Phi(u) \\ v \notin \n(\Phi(u),k)}} g(v)P(v,\Phi(u))\bigg| \\
&\stackrel{\phantom{\eqref{k_condition_2}}}{\leq} \sum_{v \in \N \times \Z^d} \sum_{\substack{u \sim v \\ u \notin \n(v,k)}} (f(v)+g(v))P(v,u) 
\stackrel{\mbox{\scriptsize\eqref{k_condition_2}}}{<} \sum_{v \in \N \times \Z^d} (f(v)+g(v))\kappa
\leq 2\kappa.
}
Next consider the first absolute value in the final line of \eqref{fixed_u}. 
The difference between the two sums can be bounded as
\eeq{ \label{individual_u}
&\bigg|\sum_{v \in \n(u,k)} f(v)P(v,u) - \sum_{v \in \n(\Phi(u),k)} g(v)P(v,\Phi(u))\bigg| \\
&\leq \bigg|\sum_{v \in \n(u,k) \cap A} f(v)P(v,u)-\sum_{v \in \n(\Phi(u),k) \cap \psi(A)} g(v)P(v,\Phi(u))\bigg| + \sum_{v \in \n(u,k) \setminus A} f(v)P(v,u) \\
&\phantom{\leq}+ \sum_{v \in \n(\Phi(u),k) \cap \psi(C\setminus A)} g(v)P(v,\Phi(u)) + \sum_{v \in \n(\Phi(u),k) \setminus \psi(C)} g(v)P(v,\Phi(u)).
}
Each of the four terms above can be controlled separately.
For the first term, notice that because $\deg(\Phi) > k$,
\eq{
v \in \n(u,k) \cap A \quad \Rightarrow \quad u - v = \Phi(u) - \Phi(v) = \Phi(u) - \psi(v).
}
That is, if $v \in \n(u,k) \cap A$, then $\psi(v) \in \n(\Phi(u),k) \cap \psi(A)$ and satisfies
\eq{
P(\psi(v),\Phi(u)) = P(\omega_1 = \Phi(u) - \psi(v)) = P(\omega_1 = u - v) = P(v,u).
}
Conversely,
\eq{
\psi(v) \in \n(\Phi(u),k) \cap \psi(A) \quad \Rightarrow \quad \Phi(u)-\psi(v) = \Phi(u)-\Phi(v) = u - v.
}
That is, if $\psi(v) \in \n(\Phi(u),k) \cap \psi(A)$, then $v \in \n(u,k) \cap A$ and satisfies
\eq{
P(v,u) = P(\omega_1 = u - v) = P(\omega_1 = \Phi(u)-\psi(v)) = P(\psi(v),\Phi(u)).
}
We thus have a bijection between $\n(u,k) \cap A$ and $\n(\Phi(u),k) \cap \psi(A)$, meaning
\eq{
\bigg|\sum_{v \in \n(u,k) \cap A} f(v)P(v,u)-\sum_{v \in \n(\Phi(u),k) \cap \psi(A)} g(v)P(v,\Phi(u))\bigg|
&= \bigg|\sum_{v \in \n(u,k) \cap A} \big[f(v) - g(\psi(v))\big]P(v,u)\bigg| \\
&\leq \sum_{v \in \n(u,k) \cap A} |f(v) - g(\psi(v))|P(v,u) \\
&\leq \sum_{v \in A}|f(v) - g(\psi(v))|P(v,u) \\
&\leq \sum_{v \in C}|f(v) - g(\psi(v))|P(v,u),
}
where the final inequality is trivial since $A \subset C$.
Summing over $u \in A^{(k)}$ gives the total bound
\eeq{ \label{all_u_1}
&\sum_{u \in A^{(k)}} \bigg|\sum_{v \in \n(u,k) \cap A} f(v)P(v,u)-\sum_{v \in \n(\Phi(u),k) \cap \psi(A)} g(v)P(v,\Phi(u))\bigg| \\
&\leq \sum_{u \in A^{(k)}} \sum_{v \in C} |f(v)-g(\psi(v))|P(v,u) \\
&= \sum_{v \in C} \sum_{u \in A^{(k)}} |f(v) - g(\psi(v))|P(v,u) 
\leq \sum_{v \in C} |f(v) - g(\psi(v))|
\leq d_{\alpha,\psi}(f,g) \stackrel{\mbox{\scriptsize\eqref{psi_condition}}}{<} \delta \stackrel{\mbox{\scriptsize\eqref{delta_condition_noA}}}{<} \kappa.
}
Considering the second term on the right-hand side of \eqref{individual_u}, we have
\eeq{ \label{all_u_2}
\sum_{u \in A^{(k)}} \sum_{v \in \n(u,k) \setminus A} f(v)P(v,u)
\leq \sum_{v \notin A} \sum_{u \in \n(v,k)} f(v)P(v,u)
\leq \sum_{v \notin A} f(v) \stackrel{\mbox{\scriptsize\eqref{A_condition}}}{<} \kappa.
}
Similarly, for the third term,
\eeq{ \label{all_u_3}
&\sum_{u \in A^{(k)}} \sum_{v \in \n(\Phi(u),k) \cap \psi(C\setminus A)} g(v)P(v,\Phi(u)) \\
&\leq \sum_{u \in A^{(k)}} \sum_{v \in \n(\Phi(u),k) \cap \psi(C\setminus A)} \big[|f(\psi^{-1}(v))-g(v)| + f(\psi^{-1}(v))\big]P(v,\Phi(u)) \\
&\leq \sum_{v \in C \setminus A} \sum_{u \in \n(\psi(v),k)}\big[|f(v)-g(\psi(v))| + f(v)\big]P(\psi(v),u) \\
&\leq \sum_{v \in C} |f(v)-g(\psi(v))| + \sum_{v \notin A} f(v)
\stackrel{\mbox{\scriptsize\eqref{A_condition}}}{\leq} d_{\alpha,\psi}(f,g) + \kappa \stackrel{\mbox{\scriptsize\eqref{psi_condition}}}{<} \delta + \kappa
\stackrel{\mbox{\scriptsize\eqref{delta_condition_noA}}}{<} 2\kappa.
}
Finally, for the fourth term,
\eeq{ \label{all_u_4}
\sum_{u \in A^{(k)}} \sum_{v \in \n(\Phi(u),k) \setminus \psi(C)} g(v)P(v,\Phi(u))
&= \sum_{u \in A^{(k)}} \sum_{v \in \n(\Phi(u),k) \setminus \psi(C)} g(v)P(0,\Phi(u)-v) \\
&\leq\sum_{u \in A^{(k)}} \sup_{v \notin \psi(C)} g(v) \\
&\leq (2d)^k|A| \bigg(\sum_{v \notin \psi(C)} g(v)^\alpha\bigg)^{1/\alpha} \\
&\leq (2d)^k|A| d_{\alpha,\psi}(f,g)^{1/\alpha} \stackrel{\mbox{\scriptsize\eqref{psi_condition}}}{<} (2d)^k|A|\delta^{1/\alpha} \stackrel{\mbox{\scriptsize\eqref{delta_condition_A}}}{<} \kappa.
}
Combining \eqref{individual_u}--\eqref{all_u_4}, we arrive at
\eeq{
\sum_{u \in A^{(k)}} \bigg|\sum_{v \in \n(u,k)} f(v)P(v,u) - \sum_{v \in \n(\Phi(u),k)} g(v)P(v,\Phi(u))\bigg| < 5\kappa. \label{all_u}
}
Using \eqref{easier_one} and \eqref{all_u} with \eqref{fixed_u} reveals 
\eq{
\sum_{u \in A^{(k)}} \EE|\wt f(u) - \wt g(\Phi(u))| \leq 7\kappa e^{\lambda(\beta)},
}
and thus we have the desired bound:
\eeq{ \label{part_2_bound}
\sum_{u \in A^{(k)}} \EE\bigg[\frac{|\wt{f}(u) - \wt{g}(\Phi(u))|}{\wt F}\bigg]
&\stackrel{\mbox{\scriptsize\eqref{initial_fkg}}}{\leq} \EE(\wt F^{-1}) \sum_{u \in A^{(k)}} \EE|\wt f(u) - \wt g(\Phi(u))| \\
&\stackrel{\mbox{\scriptsize\eqref{Fbound}}}{\leq} 14\kappa e^{\lambda(-\beta)}e^{\lambda(\beta)}
 \stackrel{\mbox{\scriptsize\eqref{kappa_condition_4}}}{<} \frac{\eps}{8\alpha}.
}

\subsubsection{Upper bound for $\EE(|\wt{G}-\wt{F}|/\wt{F})$}

First write the trivial equality
\eq{
\wt F - \wt G = (\wt F - e^{\lambda(\beta)}) - (\wt G - e^{\lambda(\beta)}),
}
and then observe that
\eeq{ \label{F_minus}
\wt F - e^{\lambda(\beta)} &= \sum_{u \in \N \times \Z^d} \wt{f}(u) - \|f\|e^{\lambda(\beta)} \\
&= \sum_{u \in \N \times \Z^d} \sum_{v \sim u} f(v) e^{\beta \eta_u}P(v,u) - \sum_{u \in \N \times \Z^d} \sum_{v \sim u} f(v)P(v,u) e^{\lambda(\beta)} \\
&= \sum_{u \in \N \times \Z^d} \sum_{v \sim u} f(v)\big[e^{\beta \eta_u} - e^{\lambda(\beta)}\big]P(v,u) \\
&= \sum_{u \in A^{(k)}} \sum_{v \sim u} f(v)\big[e^{\beta \eta_u} - e^{\lambda(\beta)}\big]P(v,u)
+ \sum_{u \notin A^{(k)}} \sum_{v \sim u} f(v)\big[e^{\beta \eta_u} - e^{\lambda(\beta)}\big]P(v,u).
}
Similarly,
\eq{
\wt{G} - e^{\lambda(\beta)}
&= \sum_{u \in \Phi(A^{(k)})} \sum_{v \sim u} g(v)\big[e^{\beta \zeta_u} - e^{\lambda(\beta)}\big]P(v,u) + \sum_{u \notin \Phi(A^{(k)})} \sum_{v \sim u} g(v)\big[e^{\beta \zeta_u} - e^{\lambda(\beta)}\big]P(v,u) \\
&= \sum_{u \in A^{(k)}} \sum_{v \sim \Phi(u)} g(v)\big[e^{\beta \eta_u} - e^{\lambda(\beta)}\big]P(v,\Phi(u)) + \sum_{u \notin \Phi(A^{(k)})} \sum_{v \sim u} g(v)\big[e^{\beta \zeta_u} - e^{\lambda(\beta)}\big]P(v,u).
}
Hence
\eq{
\wt{F} - \wt{G}
&= \sum_{u \in A^{(k)}}\bigg( \sum_{v \sim u} f(v)P(v,u) - \sum_{v \sim \Phi(u)} g(v)P(v,\Phi(u))\bigg)\big[e^{\beta \eta_u} - e^{\lambda(\beta)}\big] \\
&\phantom{=|} + \sum_{u \notin A^{(k)}} \sum_{v \sim u} f(v)\big[e^{\beta \eta_u} - e^{\lambda(\beta)}\big]P(v,u)
- \sum_{u \notin \Phi(A^{(k)})} \sum_{v \sim u} g(v)\big[e^{\beta \zeta_u} - e^{\lambda(\beta)}\big]P(v,u).
}
In the notation of Lemma \ref{LLN}, the following random variables are i.i.d.: 
\eq{
X_{u} &\coloneqq e^{\beta \eta_u} - e^{\lambda(\beta)}, \quad u \in A^{(k)} \\
X'_{u} &\coloneqq e^{\beta \eta_u} - e^{\lambda(\beta)}, \quad u \notin A^{(k)} \\
X''_{u} &\coloneqq e^{\beta \zeta_u} - e^{\lambda(\beta)}, \quad u \notin \Phi(A^{(k)}).
} 
Their coefficients are
\eq{
c_{u} &\coloneqq \sum_{v \sim u} f(v)P(v,u) - \sum_{v \sim \Phi(u)} g(v)P(v,\Phi(u)), \quad u \in A^{(k)}\\
c'_{u} &\coloneqq \sum_{v \sim u} f(v)P(v,u), \quad u \notin A^{(k)} \\
c''_{u} &\coloneqq -\sum_{v \sim u} g(v)P(v,u), \quad u \notin \Phi(A^{(k)}).
}
That is,
\begin{align}
\wt F - \wt G &= \sum_{u \in A^{(k)}} c_u X_u + \sum_{u \notin A^{(k)}} c_u'X_u' + \sum_{u \notin \Phi(A^{(k)})} c_u''X_u'' \nonumber \\
\Rightarrow \quad |\wt F - \wt G| &\leq \bigg|\sum_{u \in A^{(k)}} c_u X_u\bigg|
+ \bigg|\sum_{u \notin A^{(k)}} c'_u X_u'\bigg|
+ \bigg|\sum_{u \notin \Phi(A^{(k)})} c''_u X_u''\bigg|. \label{tilde_difference_1} 
\end{align}
Assume we can show the following inequalities:
\begin{align} 
\sum_{u \in A^{(k)}} |c_u|  \leq 7\kappa &\stackrel{\mbox{\scriptsize\eqref{kappa_condition_1}},\mbox{\scriptsize\eqref{kappa_condition_2}}}{\leq}  b \wedge 1 \label{c_sum} \\
c'_u < 2\kappa &\stackrel{\hspace{3ex}\eqref{kappa_condition_2}\hspace{3ex}}{\leq} b  \quad \text{for all $u \notin A^{(k)}$} \label{c_prime} \\
c''_u < (2\kappa)^{1/\alpha} + \kappa &\stackrel{\hspace{3ex}\eqref{kappa_condition_2}\hspace{3ex}}{\leq} b \quad \text{for all $u \notin \Phi(A^{(k)})$.} \label{c_prime_prime}
\end{align}
We also have the trivial inequalities
\eq{ 
\sum_{u \notin A^{(k)}} c_u' = \sum_{u \notin A^{(k)}} \sum_{v \sim u} f(v)P(v,u)
\leq \sum_{v \in \N \times \Z^d} \sum_{u \sim v} f(v)P(v,u)
= \|f\| \leq 1,
}
and similarly
\eq{ 
\sum_{u \notin \Phi(A^{(k)})} |c_u''| \leq \|g\| \leq 1.
}
Therefore, by the choice of $b$ in relation to Lemma \ref{LLN}, we deduce from \eqref{tilde_difference_1} that
\eeq{ \label{numerator_bound}
\EE|\wt F - \wt G| \leq 3t.
}

There are now two cases to consider.
First, if $\|f\| < 1$, then
\eeq{ \label{part_3_bound_1}
\EE\Big(\frac{|\wt F - \wt G|}{\wt F}\Big) \leq \frac{\EE|\wt F - \wt G|}{(1-\|f\|)e^{\lambda(\beta)}}
\leq \frac{3t}{(1-\|f\|)e^{\lambda(\beta)}} \stackrel{\mbox{\scriptsize\eqref{t_condition_1}}}{<} \frac{\eps}{8\alpha}.
}
On the other hand, if $\|f\| = 1$, then we can consider the three sums in \eqref{tilde_difference_1} separately.
For any $u \in A^{(k)}$, the quantity
\eq{
\wt F_u \coloneqq \wt F - \sum_{v \sim u} f(v)e^{\beta\eta_u} = \sum_{w \neq u} \sum_{v \sim w} f(v)e^{\beta \eta_w}P(v,w)
}
is independent of $X_u$.
Since
\eq{
\sum_{w \neq u} \sum_{v \sim w} f(v)P(v,w) = \sum_{v}f(v)\sum_{\substack{w \sim v \\ w \neq u}}P(v,w)
= \sum_{v}f(v)(1-P(v,u))
\geq \|f\|(1-q) = (1-q),
}
Lemma \ref{amgm} guarantees
\eq{
\EE(\wt F_u^{-1}) \leq (1-q)^{-1}e^{\lambda(-\beta)}.
}
Hence
\eeq{ \label{equals1_sum1}
\EE\bigg(\frac{\big|\sum_{u \in A^{(k)}} c_u X_u\big|}{\wt F}\bigg)
\leq \sum_{u \in A^{(k)}} |c_u|\EE\Big(\frac{|X_u|}{\wt F_u}\Big)
&\stackrel{\phantom{\eqref{c_sum}}}{=} \sum_{u \in A^{(k)}} |c_u|\EE|X_u|\EE(\wt F_u^{-1}) \\
&\stackrel{\mbox{\scriptsize\eqref{c_sum}}}{\leq} 7\kappa(1-q)^{-1}e^{\lambda(-\beta)}\EE|e^{\beta\eta}-e^{\lambda(\beta)}|.
}
We must also have
\eq{
\sum_{u \in A^{(k)}} \sum_{v \sim u} f(v) P(v,u)
\geq \sum_{v \in A} \sum_{u \in \n(v,k)} f(v) P(v,u) 
\stackrel{\mbox{\scriptsize\eqref{k_condition_2}}}{>} (1-\kappa)\sum_{v \in A} f(v) \stackrel{\mbox{\scriptsize\eqref{A_condition}}}{>} (1-\kappa)^2,
}
and so by applying Lemma \ref{amgm} once more, we see
\eeq{ \label{equals1_sum2}
\EE\bigg[\bigg(\sum_{u \in A^{(k)}} \sum_{v \sim u} f(v)e^{\beta \eta_u}P(v,u)\bigg)^{-1}\bigg] \leq (1-\kappa)^{-2}e^{\lambda(-\beta)}.
}
Again appealing to independence and then Lemma \ref{LLN}, we obtain the bound
\eeq{ \label{equals1_sum3}
\EE\bigg(\frac{\big|\sum_{u \notin A^{(k)}} c_u'X_u'\big|}{\wt F}\bigg)
&\leq \EE\bigg(\frac{\big|\sum_{u \notin A^{(k)}} c_u'X_u'\big|}{\sum_{u \in A^{(k)}} \sum_{v \sim u} f(v)e^{\beta \eta_u}P(v,u)}\bigg) \\
&=\EE\bigg|\sum_{u \notin A^{(k)}} c_u'X_u'\bigg|\EE\bigg[\bigg(\sum_{u \in A^{(k)}} \sum_{v \sim u} f(v)e^{\beta \eta_u}P(v,u)\bigg)^{-1}\bigg]
\leq t(1-\kappa)^{-2}e^{\lambda(-\beta)}.
}
Finally, we use independence and Lemma \ref{LLN} once more to obtain
\eq{
\EE\bigg(\frac{\big|\sum_{u \notin \Phi(A^{(k)})} c_u''X_u''\big|}{\wt F}\bigg)
= \EE\bigg|\sum_{u \notin \Phi(A^{(k)})} c_u''X_u''\bigg|\EE(\wt F^{-1})
\stackrel{\mbox{\scriptsize\eqref{Fcase_2}}}{\leq} te^{\lambda(-\beta)}.
}
Combining \eqref{equals1_sum1}--\eqref{equals1_sum3}, we conclude that if $\|f\| = 1$, then
\eeq{ \label{part_3_bound_2}
\EE\Big(\frac{|\wt F - \wt G|}{\wt F}\Big) 
&\stackrel{\phantom{\eqref{kappa_condition_1},\eqref{kappa_condition_3}}}{\leq} 7\kappa(1-q)^{-1}e^{\lambda(-\beta)}\EE|e^{\beta\eta}-e^{\lambda(\beta)}| + t(1-\kappa)^{-2}e^{\lambda(-\beta)} + te^{\lambda(-\beta)} \\
&\stackrel{\phantom{xxx}\mbox{\scriptsize{\eqref{kappa_condition_1}}}\phantom{xxx}}{<} 7\kappa(1-q)^{-1}e^{\lambda(-\beta)}\EE|e^{\beta\eta}-e^{\lambda(\beta)}| + \frac{85t}{36}e^{\lambda(-\beta)} \\
&\stackrel{\mbox{\scriptsize\eqref{kappa_condition_3},\eqref{t_condition_2}}}{<} \frac{\eps}{8\alpha}.
}
Having established \eqref{part_3_bound_1} and \eqref{part_3_bound_2}, we may write
\eeq{ \label{part_3_bound}
\EE\Big(\frac{|\wt F - \wt G|}{\wt F}\Big) < \frac{\eps}{8\alpha},
}
irrespective of the value of $\|f\|$.

Now we check the inequalities \eqref{c_sum}--\eqref{c_prime_prime}.
First, for the $c_u$,
\eq{
\sum_{u \in A^{(k)}} |c_u| 
&\leq \sum_{u \in A^{(k)}} \bigg|\sum_{v \in \n(u,k)} f(v)P(v,u) - \sum_{v \in \n(\Phi(u),k)} g(v)P(v,\Phi(u))\bigg| \\
&\phantom{\leq}+ \sum_{u \in A^{(k)}}\bigg|\sum_{\substack{v\sim u \\ v \notin \n(u,k)}} f(v)P(v,u) - \sum_{\substack{v \sim \Phi(u) \\ v \notin \n(\Phi(u),k)}} g(v)P(v,\Phi(u))\bigg|,
}
where the first sum is bounded by $5\kappa$ according to $\eqref{all_u}$,
and the second sum satisfies
\eq{
\sum_{u \in A^{(k)}}\bigg|\sum_{\substack{v\sim u \\ v \notin \n(u,k)}} f(v)P(v,u) - \sum_{\substack{v \sim \Phi(u) \\ v \notin \n(\Phi(u),k)}} g(v)P(v,\Phi(u))\bigg|
&\stackrel{\phantom{\eqref{k_condition_2}}}{\leq} \sum_{u \in \N\times\Z^d}\sum_{\substack{v\sim u \\ v \notin \n(u,k)}} (f(v)+g(v))P(v,u) \\
&\stackrel{\phantom{\eqref{k_condition_2}}}{=}\sum_{v \in \N\times\Z^d}\sum_{\substack{u\sim v \\ u \notin \n(v,k)}} (f(v)+g(v))P(v,u) \\
&\stackrel{\mbox{\scriptsize\eqref{k_condition_2}}}{<}\sum_{v\in\N\times\Z^d} (f(v)+g(v))\kappa \leq 2\kappa.
}
Hence \eqref{c_sum} holds:
\eq{
\sum_{u \in A^{(k)}} |c_k| \leq 5\kappa + 2\kappa = 7\kappa.
}
Next consider the $c_u'$.
By definition of $A^{(k)}$,
\eeq{ \label{notinAk_implication}
u \notin A^{(k)} \quad \Rightarrow \quad \n(u,k) \subset (\N \times \Z^d) \setminus A,
}
which implies
\eq{
c'_u &=  \sum_{v \in \n(u,k)} f(v)P(v,u) + \sum_{\substack{v \sim u \\ v \notin \n(u,k)}} f(v)P(v,u) \\
&\leq \sum_{v \notin A} f(v)P(v,u) + P(\|\omega_1\|_1 > k)
\stackrel{\mbox{\scriptsize\eqref{k_condition_2}}}{<} \sum_{v \notin A} f(v) + \kappa
\stackrel{\mbox{\scriptsize\eqref{A_condition}}}{<} 2\kappa < 7\kappa.
}
Last, consider the $c_u''$.
We have the implication
\eeq{ \label{notinPhiAk_implication}
u \notin \Phi(A^{(k)}) \stackrel{\mbox{\scriptsize\eqref{degree_consequence}}}{=} \Phi(A)^{(k)} \quad \Rightarrow \quad \n(u,k) \subset (\N\times\Z^d)\setminus \Phi(A),
}
and thus
\eeq{ \label{c_prime_prime_setup}
|c_u''| &= \sum_{v \in \n(u,k)}g(v)P(v,u) + \sum_{\substack{v\sim u \\ v \notin \n(u,k)}} g(v)P(v,u) \\
&\leq \sum_{v \notin \Phi(A)} g(v)P(v,u) + \sum_{\substack{v\sim u \\ v \notin \n(u,k)}} g(v)P(0,u-v) 
\leq \sup_{v \notin \Phi(A)} g(v) + P(\|\omega_1\|_1 > k).
}
A further bound is needed:
\eeq{ \label{notinPhiA_sum}
\sum_{v \notin \Phi(A)} g(v)^\alpha
&\leq \sum_{v \notin \psi(C)} g(v)^\alpha + \sum_{v \in \psi(C \setminus A)} g(v) \\
&= \sum_{v \notin \psi(C)} g(v)^\alpha + \sum_{v \in C \setminus A} g(\psi(v)) \\
&\leq \sum_{v \notin \psi(C)} g(v)^\alpha + \sum_{v \in C \setminus A} |f(v)-g(\psi(v))| + \sum_{v \in C \setminus A} f(v) \\
&\leq \sum_{v \notin \psi(C)} g(v)^\alpha + \sum_{v \in C} |f(v) - g(\psi(v))| + \sum_{v \notin A} f(v) \\
&\stackrel{\mbox{\scriptsize\eqref{A_condition}}}{<} d_{\alpha,\psi}(f,g) + \kappa 
\stackrel{\mbox{\scriptsize\eqref{psi_condition}}}{<} \delta + \kappa
\stackrel{\mbox{\scriptsize\eqref{delta_condition_noA}}}{<} 2\kappa.
}
In particular,
\eq{
\sup_{v \in \Phi(A)} g(v) \leq \bigg(\sum_{v \notin \Phi(A)} g(v)^\alpha\bigg)^{1/\alpha} < (2\kappa)^{1/\alpha},
}
which completes \eqref{c_prime_prime_setup} to yield \eqref{c_prime_prime}:
\eq{ 
|c_u''|
\leq \sup_{v \notin \Phi(A)} g(v) + P(\|\omega_1\|_1 > k)
\stackrel{\mbox{\scriptsize\eqref{k_condition_2}}}{\leq} (2\kappa)^{1/\alpha} + \kappa.
}

\subsubsection{Upper bound for $\EE\big[\sum_{u \notin A^{(k)}} \wt f(u)^\alpha\big]$ and $\EE\big[\sum_{u \notin \Phi(A^{(k)})} \wt g(u)^\alpha\big]$}
By definition,
\eq{
\EE\bigg[\sum_{u \notin A^{(k)}} \wt f(u)^\alpha\bigg]
= \sum_{u \notin A^{(k)}} \EE\bigg(\sum_{v \sim u} f(v)e^{\beta\eta_u}P(v,u)\bigg)^\alpha
= e^{\lambda(\alpha\beta)}\sum_{u \notin A^{(k)}} \bigg(\sum_{v\sim u} f(v)P(v,u)\bigg)^\alpha.
}
Since the last inner sum is at most 1, we have
\eq{
\EE\bigg[\sum_{u \notin A^{(k)}} \wt f(u)^\alpha\bigg]
\leq e^{\lambda(\alpha\beta)}\sum_{u \notin A^{(k)}} \sum_{v\sim u} f(v)P(v,u),
}
where
\eq{
\sum_{u \notin A^{(k)}} \sum_{v\sim u} f(v)P(v,u)
&\stackrel{\phantom{\eqref{notinAk_implication}}}{=} \sum_{u \notin A^{(k)}}\bigg[ \sum_{v\in \n(u,k)} f(v)P(v,u) + \sum_{\substack{v \sim u \\ v \notin \n(u,k)}} f(v)P(v,u)\bigg] \\
&\stackrel{\mbox{\scriptsize\eqref{notinAk_implication}}}{\leq} \sum_{v \notin A} \sum_{u \sim v} f(v)P(v,u) + \sum_{v \in \N \times \Z^d} \sum_{\substack{u \sim v \\ u \notin \n(v,k)}} f(v)P(v,u) \\
&\stackrel{\phantom{\eqref{notinAk_implication}}}{=} \sum_{v \notin A} f(v) + \sum_{v \in \N \times \Z^d} f(v)P(\|\omega_1\|_1 > k)
\stackrel{\mbox{\scriptsize\eqref{A_condition},\eqref{k_condition_2}}}{<} 2\kappa.
}
Combining the two previous two displays yields
\eeq{
\EE\bigg[\sum_{u \notin A^{(k)}} \wt f(u)^\alpha\bigg] < 2\kappa e^{\lambda(\alpha\beta)}. \label{part_4_bound1}
}
For $\wt g$, Jensen's inequality gives
\eq{
\EE\bigg[\sum_{u \notin \Phi(A^{(k))}} \wt g(u)^\alpha\bigg]
= e^{\lambda(\alpha\beta)} \sum_{u \notin \Phi(A^{(k)})} \bigg(\sum_{v \sim u} g(v)P(v,u)\bigg)^\alpha
\leq e^{\lambda(\alpha\beta)} \sum_{u \notin \Phi(A^{(k)})} \bigg(\sum_{v \sim u} g(v)^\alpha P(v,u)\bigg),
}
where
\eq{
 \sum_{u \notin \Phi(A^{(k)})} \bigg(\sum_{v \sim u} g(v)^\alpha P(v,u)\bigg)
 &\stackrel{\phantom{\eqref{k_condition_2}}}{=}  \sum_{u \notin \Phi(A^{(k)})}\bigg[\sum_{v \in \n(u,k)} g(v)^\alpha P(v,u) + \sum_{\substack{v \sim u \\ v \notin \n(u,k)}} g(v)^\alpha P(v,u)\bigg] \\
 &\stackrel{\mbox{\scriptsize\eqref{notinPhiAk_implication}\phantom{b}}}{\leq} \sum_{v \notin \Phi(A)}\sum_{u \in \n(u,k)} g(v)^\alpha P(v,u) + \sum_{v \in \N \times \Z^d}\sum_{\substack{u \sim v \\ u \notin \n(v,k)}} g(v)^\alpha P(v,u) \\
 &\stackrel{\mbox{\scriptsize\eqref{k_condition_2}}}{<} \sum_{v \notin \Phi(A)}g(v)^\alpha + \kappa
 \stackrel{\mbox{\scriptsize\eqref{notinPhiA_sum}}}{<} 3\kappa.
}
In summary,
\eeq{ \label{part_4_bound2}
\EE\bigg[\sum_{u \notin \Phi(A^{(k))}} \wt g(u)^\alpha\bigg] < 3\kappa e^{\lambda(\alpha\beta)}.
}
\subsubsection{Completing the proof}
Once \eqref{part_2_bound} and \eqref{part_3_bound} are used in \eqref{first_part} and \eqref{second_part}, respectively, \eqref{two_parts} gives
\eeq{ \label{final_sum_1}
\EE\bigg[\alpha\sum_{u \in A^{(k)}} |F(u)-G(\Phi(u))|\bigg] < \alpha\Big(\frac{\eps}{8\alpha}+\frac{\eps}{8\alpha}\Big) = \frac{\eps}{4}.
}
Next, \eqref{Fbound} and \eqref{part_4_bound1} make \eqref{to_bound_2} read as
\eeq{ \label{final_sum_2}
\EE\bigg[\sum_{u \notin A^{(k)}} F(u)^\alpha\bigg] < 2^{\alpha+1} e^{\lambda(\alpha\beta)+\lambda(-\alpha\beta)}\kappa \stackrel{\mbox{\scriptsize\eqref{kappa_condition_5}}}{<} \frac{\eps}{4}.
}
Similarly, \eqref{Gbound} and \eqref{part_4_bound2} translate \eqref{to_bound_3} to
\eeq{ \label{final_sum_3}
\EE\bigg[\sum_{u \notin \Phi(A^{(k))}} G(u)^\alpha\bigg] < 3\cdot 2^{\alpha}e^{\lambda(\alpha\beta)+\lambda(-\alpha\beta)}\kappa \stackrel{\mbox{\scriptsize\eqref{kappa_condition_5}}}{<} \frac{\eps}{4}.
}
Finally,
\eeq{ \label{final_sum_4}
\deg(\Phi) > k \quad {\Rightarrow} \quad 2^{-\deg(\Phi)} < 2^{-k} \stackrel{\mbox{\scriptsize\eqref{k_condition_1}}}{<} \frac{\eps}{4}.
}
Together, \eqref{final_sum_1}--\eqref{final_sum_4} produce the desired result:
\eq{
\EE(d_{\alpha,\Phi}(F,G)) < \eps.
}
\end{proof}

\section{Variational formula for free energy} \label{varitional_formula_section}
Given the results in Sections \ref{free_energy_background} and \ref{adaptation_update_map}, there is little difficulty left in obtaining the main results.
Indeed, the majority of remaining proofs go through as in \cite{bates-chatterjee17} with no change.
The main reason for this is the high level of abstraction in our approach, the essential components of which are $\s$ and $\t$.
Since $\t$ can be thought of as a continuous transition kernel for a Markov chain in a compact space $\s$, general ergodic theory provides swift access to results on Ces\`aro limits.
Furthermore, the high and low temperature phases can be characterized using the variational formula \eqref{variational_formula} of Theorem \ref{free_energy_thm}.

\subsection{Outline of abstract methods} \label{outline_methods}
Let us now recall a progression of results, with appropriate references to \cite{bates-chatterjee17}.
Whenever modification is necessary, an updated proof is provided in Section \ref{updated_proofs}.
Recall from \eqref{fn_def} that $f_n$ is the quenched endpoint distribution of the polymer at length $n$, identified as an element of $\s$.
\subsubsection{Step 1} Define the empirical measure
\eeq{
\mu_n \coloneqq \frac{1}{n+1} \sum_{i = 0}^{n} \delta_{f_i}, \label{mu_n_def}
}
which is a random element of $\p(\s)$, measurable with respect to $\f_n$.
Since $Tf_n$ is the law of $f_{n+1}$ given $\f_n$, a martingale argument shows that $\mu_n$ almost surely converges to the set of fixed points of $\t$, which we denote
\eeq{ \label{K_def}
\k \coloneqq \{\nu \in \p(\s): \t\nu = \nu\}.
}
More precisely, if for $\mu \in \p(\s)$ and $\u \subset \p(\s)$ we define
\eq{
\w_\alpha(\mu,\u) \coloneqq \inf_{\nu \in \u} \w_\alpha(\mu,\nu),
}
then 
\eeq{ \label{converge_to_K}
\lim_{n \to \infty} \w_\alpha(\mu_n,\k) = 0 \quad \mathrm{a.s.}
}
This is the content of \cite[Sections 4.1 and 4.2]{bates-chatterjee17}, from which all the proofs remain valid with $\t$ defined for a general $P$.

\subsubsection{Step 2}
As in \cite[Section 4.3]{bates-chatterjee17}, define the energy functional $R : \s \to \R$ by
\eeq{
R(f) \coloneqq \EE \log \bigg(\sum_{u \in \N \times \Z^d} \sum_{v \sim u} f(v)e^{\beta \eta_u}P(v,u) + (1-\|f\|)e^{\lambda(\beta)}\bigg). \label{r_def}
}
That $R$ is well-defined and continuous is the content of \cite[Lemma 4.5]{bates-chatterjee17}; the proof there relies only on facts we have already adapted in the course of proving Proposition \ref{continuous1}.
When $f = f_n$ is the $n$-th endpoint distribution, \eqref{ratio_equality} shows
\eq{
R(f_n) = \EE\givenk[\Big]{\log \frac{Z_{n+1}}{Z_n}}{\f_n}.
}
Therefore,
\eq{
\EE(R(f_n)) = \EE\Big(\log \frac{Z_{n+1}}{Z_n}\Big) \quad \Rightarrow \quad
\frac{1}{n}\sum_{i = 0}^{n-1} \EE(R(f_i)) = \frac{1}{n}\sum_{i = 0}^{n-1} \EE\Big( \log \frac{Z_{i+1}}{Z_i}\Big) = \EE(F_n).
}
Upon lifting $R$ to the map $\r : \p(\s) \to \R$ given by
\eeq{
\r(\mu) \coloneqq \int_\s R(f)\ \mu(\dd f), \label{R_def}
}
we obtain a continuous functional on measures by Lemma \ref{portmanteau}.
Furthermore, the above calculation can be rewritten as $\EE(\r(\mu_{n-1})) = \EE(F_n)$.
It is proved in \cite[Proposition 4.6]{bates-chatterjee17} that in fact $\r(\mu_{n-1}) - F_n \to 0$ almost surely as $n \to \infty$, and so Proposition \ref{free_energy_converges} implies
\eeq{
p(\beta) = \lim_{n \to \infty} \r(\mu_n) = \lim_{n \to \infty} \EE(\r(\mu_n)). \label{p_beta_identity}
}
The only part of the proof that requires modification is stated as Lemma \ref{fourth_moment_lemma} in the next section.

\subsubsection{Step 3}
To make use the identity \eqref{p_beta_identity}, we show in Theorem \ref{free_energy_thm} that
\eeq{
\lim_{n \to \infty} \EE(\r(\mu_n)) = \inf_{\nu \in \k} \r(\nu). \label{variational_formula_expectation}
}
The ``$\geq$" direction easily follows from \eqref{converge_to_K} (see \cite[Proposition 4.8]{bates-chatterjee17}), while the reverse inequality requires an argument we make more general in Lemma \ref{upper_bound_lemma}.
Together, \eqref{p_beta_identity} and \eqref{variational_formula_expectation} yield the variational formula for the free energy.

\begin{thm} \label{free_energy_thm}
Assume \eqref{mgf_assumption}. Then
\eeq{
p(\beta) = \inf_{\nu \in \k} \r(\nu). \label{variational_formula}
}
Furthermore,
\eeq{
\lim_{n \to \infty} \w_\alpha(\mu_n,\m) = 0 \quad \mathrm{a.s.}, \label{convergence_to_M}
}
where
\eeq{ \label{M_def}
\m \coloneqq \Big\{\nu_0 \in \k : \r(\nu_0) = \inf_{\nu \in \k} \r(\nu)\Big\}.
}
\end{thm}
We note that \eqref{convergence_to_M} follows from \eqref{variational_formula} exactly as in \cite[Theorem 4.11]{bates-chatterjee17}.

\subsection{Proofs in general setting} \label{updated_proofs}
Once the lemmas of this section are checked, all of the results stated in Section \ref{outline_methods} will be proved.

\subsubsection{Adaptation of a fourth moment bound}
The proof of \cite[Proposition 4.6]{bates-chatterjee17} uses the Burkholder-Davis-Gundy (BDG) inequality \cite[Theorem 1.1]{burkholder-davis-gundy72} applied to a martingale whose differences are of form $W - \EE(W)$,
where $W$ is a random variable defined using a fixed $f \in \s_0$ with $\|f\| = 1$.
Specifically,
\eeq{ \label{random_variable_W}
W \coloneqq \log \bigg(\sum_{u \in \N \times \Z^d} \sum_{v\sim u} f(v)e^{\beta \eta_u}P(v,u)\bigg),
}
where the $\eta_u$ are i.i.d.~random variables with law $\mathfrak{L}_\eta$.
\cite[Lemma 4.7]{bates-chatterjee17}) established that there is a constant not depending on $f$ such that
\eq{
\EE\big[(W-\EE\, W)^4\big] \leq C,
}
thus making the BDG inequality useful, but the proof assumed $\lambda(\pm 2\beta) < \infty$.
Here we make a simple adaptation of the proof that assumes only $\lambda(\pm \beta) < \infty$ and obtains a different value for $C$.

\begin{lemma} \label{fourth_moment_lemma}
Fix $f \in \s_0$ with $\|f\| = 1$.
Define $W$ by \eqref{random_variable_W}.
Then there is a constant $C$ not depending on $f$ such that
\eq{
\EE\big[(W-\EE\, W)^4\big] \leq C.
}
\end{lemma}

\begin{proof}
We start with the bound
\eq{
\EE\big[(W - \EE\, W)^4\big] &\leq \EE\big[(|W|+|\EE W|)^4\big] \leq 16\big[\EE(W^4) + |\EE W|^4\big] \leq 32\EE(W^4).
}
It thus suffices to show that $\EE(W^4)$ is bounded by a constant not depending on $f$.
To simplify notation, we introduce the variable
\eq{
\wt F \coloneqq \sum_{u \in \N \times \Z^d} \sum_{v \sim u} f(v) e^{\beta \eta_u}P(v,u).
}
We have
\eq{
\EE(W^{4}) &= \EE\Givenp{W^{4}}{\wt F< 1} + \EE\Givenp{W^{4}}{\wt F \geq 1} 
=  \EE\Givenp{\log^{4}(\wt F)}{\wt F < 1} + \EE\Givenp{\log^{4}(\wt F)}{\wt F \geq 1}.
}
Notice that $\log^{4}(t) \leq 5t^{-1}$ for all $t \in (0,1)$, while $\log^{4}(t) \leq 5t$ for all $t \in [1,\infty)$.
Since $\wt F$ is positive, these observations lead to
\eq{
\EE(W^{4}) &\leq 5\EE\Givenp{\wt F^{-1}}{\wt F < 1} + 5\EE\Givenp{\wt F}{\wt F \geq 1}
\leq 5\EE(\wt F^{-1}) + 5\EE(\wt F) 
\leq 5(e^{\lambda(-\beta)} + e^{\lambda(\beta)}),
}
where the final inequality is due to \eqref{Fcase_2}.

\end{proof}

\subsubsection{Adaptation of a free energy inequality}
The proof of \eqref{variational_formula} can be written exactly as the proof of Theorem 4.9 in \cite{bates-chatterjee17}, but it requires the result of the next lemma.
We introduce the boldface notation $\vc 1$ to denote the element of $\s$ having representatives in $\s_0$ of the form
\eq{
f(u) = \begin{cases}
1 &\text{if }u = u_0 \\
0 &\text{otherwise,}
\end{cases} \quad u \in \N\times\Z^d.
}
Similarly, $\vc 0$ will denote the element of $\s$ whose unique representative is the constant zero function.

\begin{lemma} \label{upper_bound_lemma}
For any $f_0 \in \s$ and $n \geq 1$,
\eq{
\sum_{i = 0}^{n-1} \r(\t^i \delta_{f_0}) \geq \EE(\log Z_n),
}
where $\delta_{f_0} \in \p(\s)$ is the unit mass at $f_0$.
Equality holds if and only if $f_0 = \vc{1}$.
\end{lemma}

\begin{proof}
Fix a representative $f_0 \in \s_0$.
Let $(\eta^{(i)}_u)_{u \in \N \times \Z^d}$, $1 \leq i \leq n$, be independent collections of i.i.d.~random variables with law $\mathfrak{L}_\eta$.
For $1 \leq i \leq n$, inductively define $f_i \in \s$ to have representative
\eq{
f_i(u) &= \frac{\sum_{v \sim u} f_{i-1}(v) \exp(\beta \eta^{(i)}_u)P(v,u)}{ \sum_{w \in \N \times \Z^d}\sum_{v \sim w} f_{i-1}(v) \exp(\beta \eta^{(i)}_w)P(v,w) + (1-\|f_{i-1}\|)e^{\lambda(\beta)}} \\
&= \frac{\sum_{v \sim u} f_{i-1}(v) \exp(\beta \eta^{(i)}_u)P(v,u)}{D_i},
}
where
\eeq{ \label{mess0}
D_i \coloneqq \sum_{u_i \in \N \times \Z^d}\sum_{u_{i-1} \sim u_i} f_{i-1}(u_{i-1})\exp(\beta \eta^{(i)}_{u_i})P(u_{i-1},u_{i}) + (1 - \|f_{i-1}\|)e^{\lambda(\beta)}, \quad 1 \leq i \leq n.
}
By construction, the law of $f_i$ is equal to $\t$ applied to the law of $f_{i-1}$.
Upon making the inductive assumption that the law of $f_{i-1}$ is $\t^{i-1}\delta_{f_0}$, we see that the law of $f_i$ is $\t^i \delta_{f_0}$.
By definitions \eqref{r_def} and \eqref{R_def},
\eeq{ \label{mess0_2}
\r(\t^i \delta_{f_0}) &= \EE\log\bigg(\sum_{u_{i+1} \in \N \times \Z^d} \sum_{u_i \sim u_{i+1}} f_i(u_i)\exp(\beta \eta^{(i+1)}_{u_{i+1}})P(u_{i},u_{i+1}) + (1-\|f_i\|)e^{\lambda(\beta)}\bigg) \\
&= \EE(\log D_{i+1}).
}
Observe that when $i=n$, the first summand in \eqref{mess0} is equal to
\eeq{
&\sum_{u_{n} \in \N \times \Z^d} \sum_{u_{n-1} \sim u_{n}} f_{n-1}(u_{n-1})\exp(\beta \eta^{(n)}_{u_{n}})P(u_{n-1},u_{n})  \\
&= \frac{1}{D_{n-1}}\sum_{u_{n} \in \N \times \Z^d} \sum_{u_{n-2} \sim u_{n-1} \sim u_{n}} f_{n-2}(u_{n-2})\exp(\beta \eta^{(n-1)}_{u_{n-1}} + \beta \eta^{(n)}_{u_{n}})P(u_{n-2},u_{n-1})P(u_{n-1},u_{n})  \\
&\hspace{0.08in}\vdots \\
&= \frac{1}{D_1D_2\cdots D_{n-1}}\sum_{u_{n} \in \N \times \Z^d} \sum_{u_{0} \sim u_{1} \sim \cdots \sim u_{n}} f_{0}(u_{0})\exp\bigg(\beta \sum_{i = 1}^{n} \eta^{(i)}_{u_{i}}\bigg)\prod_{i=1}^{n}P(u_{i-1},u_{i})  \\
&= \frac{1}{D_1D_2\cdots D_{n-1}}\sum_{u_0 \in \N \times \Z^d} \sum_{u_{n} \sim u_{n-1} \sim \cdots \sim u_0} f_0(u_0)\exp\bigg(\beta \sum_{i = 1}^{n} \eta^{(i)}_{u_{i}}\bigg)\prod_{i=1}^{n}P(u_{i-1},u_{i})  \\
&= \frac{1}{D_1D_2\cdots D_{n-1}}\sum_{u_0 \in \N \times \Z^d} \frac{f_0(u_0)}{\|f_0\|}\, \|f_0\|\sum_{u_{n} \sim u_{n-1} \sim \cdots \sim u_0} \exp\bigg(\beta \sum_{i = 1}^{n} \eta^{(i)}_{u_{i}}\bigg)\prod_{i=1}^{n}P(u_{i-1},u_{i}) \label{mess1},
}
while the second summand in \eqref{mess0} is
\eeq{
(1 - \|f_{n-1}\|)e^{\lambda(\beta)}
&= \bigg(1 - \frac{1}{D_{n-1}}\sum_{u_{n-1} \in \N \times \Z^d}\sum_{u_{n-2} \sim u_{n-1}} f_{n-2} (u_{n-2})\exp(\beta \eta^{(n-1)}_{u_{n-1}})P(u_{n-2},u_{n-1})\bigg)e^{\lambda(\beta)}   \\
&= \frac{1}{D_{n-1}}\Big[D_{n-1} - \big(D_{n-1}-(1-\|f_{n-2}\|)e^{\lambda(\beta)}\big)\Big]e^{\lambda(\beta)}  \\
&= \frac{(1 - \|f_{n-2}\|)e^{2\lambda(\beta)}}{D_{n-1}}   \\
&= \frac{(1 - \|f_{n-3}\|)e^{3\lambda(\beta)}}{D_{n-2}D_{n-1}} 
=\cdots=\frac{(1 - \|f_{0}\|)e^{n\lambda(\beta)}}{D_1D_2 \cdots D_{n-1}}
= \frac{(1-\|f_0\|)\, \EE(Z_{n})}{D_1D_2\cdots D_{n-1}}. \label{mess2} \raisetag{4\baselineskip}
}
By summing the final expressions in \eqref{mess1} and \eqref{mess2} to obtain the right-hand side of \eqref{mess0}, and then clearing the fraction, we see
\eq{
&D_1D_2\cdots D_{n-1}D_{n} \\
&= \sum_{u_0 \in \N \times \Z^d} \frac{f_0(u_0)}{\|f_0\|}\Bigg[\|f_0\|\Bigg(\sum_{u_{n} \sim u_{n-1} \sim \cdots \sim u_0}\exp\bigg(\beta \sum_{i = 1}^n \eta^{(i)}_{u_i}\bigg)\prod_{i=1}^{n}P(u_{i-1},u_{i})\Bigg) + (1 - \|f_0\|)\EE(Z_{n})\Bigg].
}
Using the concavity of the $\log$ function, we further deduce
\eq{
&\log D_1D_2\cdots D_n \\
&\geq 
\sum_{u_0 \in \N \times \Z^d} \frac{f_0(u_0)}{\|f_0\|} \log \Bigg[\|f_0\|\Bigg(\sum_{u_{n} \sim u_{n-1} \sim \cdots \sim u_0}\exp\bigg(\beta \sum_{i = 1}^n \eta^{(i)}_{u_i}\bigg)\prod_{i=1}^{n}P(u_{i-1},u_{i})\Bigg) + (1 - \|f_0\|)\EE(Z_{n})\Bigg] \\
&\geq \sum_{u_0 \in \N \times \Z^d} \frac{f_0(u_0)}{\|f_0\|}
\Bigg[\|f_0\|\log \Bigg(\sum_{u_{n} \sim u_{n-1} \sim \cdots \sim u_0}\exp\bigg(\beta \sum_{i = 1}^n \eta^{(i)}_{u_i}\bigg)\prod_{i=1}^{n}P(u_{i-1},u_{i})\Bigg) + (1 - \|f_0\|)\log \EE(Z_{n})\Bigg] \\
&\geq \sum_{u_0 \in \N \times \Z^d} \frac{f_0(u_0)}{\|f_0\|}
\Bigg[\|f_0\|\log \Bigg(\sum_{u_{n} \sim u_{n-1} \sim \cdots \sim u_0}\exp\bigg(\beta \sum_{i = 1}^n \eta^{(i)}_{u_i}\bigg)\prod_{i=1}^{n}P(u_{i-1},u_{i})\Bigg) + (1 - \|f_0\|)\EE(\log Z_n)\Bigg],
}
where equality holds throughout if and only if $f_0(u_0) = 1$ for some $u_0 \in \Z^d$.
Since the random variable $\sum_{u_{n} \sim u_{n-1} \sim \cdots \sim u_0}\exp\big(\beta \sum_{i = 1}^n \eta^{(i)}_{u_i}\big)\prod_{i=1}^{n}P(u_{i-1},u_{i})$ is equal in law to $Z_n$ for any fixed $u_0 \in \N \times \Z^d$, taking expectation yields
\eq{
\EE(\log D_1D_2\cdots D_n) &\geq \sum_{u_0 \in \N \times \Z^d} \frac{f_0(u_0)}{\|f_0\|} \Big(\|f_0\|\, \EE(\log Z_n) + (1 - \|f_0\|)\, \EE(\log Z_n)\Big)
= \EE(\log Z_n).
}
It follows that
\eq{
\sum_{i = 0}^{n-1}  \r(\t^i \delta_{f_0})
\stackrel{\mbox{\scriptsize\eqref{mess0_2}}}{=} \sum_{i = 0}^{n-1}\EE(\log D_{i+1})
&= \EE(\log D_1D_2\cdots D_n)
\geq \EE(\log Z_n),
}
with equality if and only if $f_0 = \vc{1}$.
\end{proof}

\section{Proofs of main results} \label{proof_main_results}
In this brief final section we prove Theorems \ref{apa_new} and \ref{geometric_new}.
As before, $\vc 0$ is the element of $\s$ whose unique representative in $\s_0$ is the zero function.
Recall the sets $\k$ and $\m$ defined by \eqref{K_def} and \eqref{M_def}, respectively.

\begin{thm}[{cf.~\cite[Theorem 5.2]{bates-chatterjee17}}] \label{characterization}
Assume \eqref{mgf_assumption}.
Then the following statements hold:
\begin{itemize}
\item[(a)] If $0 \leq \beta \leq \beta_c$, then $\k = \m = \{\delta_{\vc{0}}\}$.
\item[(b)] If $\beta > \beta_c$, then $\nu(\{f \in \s : \|f\| = 1\}) = 1$ for every $\nu \in \m$, and so $\t$ has more than one fixed point.
\end{itemize}
\end{thm}

\begin{remark}
Note that $\delta_{\vc 0}$ is always a fixed point of $\t$.
\end{remark}

\begin{proof}
Given $f \in \s$, consider the random variable $F \in \s$ having representative defined by \eqref{F_def}.
First observe that $\|F\| = 0$ if and only if $\|f\| = 0$, and similarly $\|F\| = 1$ if and only if $\|f\| = 1$.
On the other hand, if $f \in \s$ satisfies $0 < \|f\| < 1$, then Jensen's inequality applied to the concave function $t \mapsto t/(t+(1-\|f\|e^{\lambda(\beta)})$ implies
\eq{
\int \|F\|\ \t\delta_{f}(\dd F)
&= \EE\bigg[\frac{\sum_{u \in \N \times \Z^d} \sum_{v \sim u} f(v)e^{\beta \eta_u}P(v,u)}{\sum_{u \in \N \times \Z^d} \sum_{v\sim u} f(v)e^{\beta \eta_u}P(v,u) + (1-\|f\|)e^{\lambda(\beta)}}\bigg] \\
&< \frac{\EE\big[\sum_{u \in \N \times \Z^d} \sum_{v \sim u} f(v)e^{\beta \eta_u}P(v,u)\big]}{\EE\big[\sum_{u \in \N \times \Z^d} \sum_{v\sim u} f(v)e^{\beta \eta_u}P(v,u)\big] + (1-\|f\|)e^{\lambda(\beta)}} \\
&= \frac{\|f\|e^{\lambda(\beta)}}{\|f\|e^{\lambda(\beta)}+(1-\|f\|)e^{\lambda(\beta)}} = \|f\|.
}
The inequality above is strict because the relevant concave function is not linear, and its argument is not an almost sure constant.
It follows that if $\nu \in \k$, then $\nu(\{f \in \s : 0 < \|f\| < 1\}) = 0$.
Furthermore, since $\|F\| = \|f\|$ when $\|f\| \in \{0,1\}$, we may conclude that any $\nu \in \k$ decomposes as a convex combination
\eq{
\nu = t\delta_{\vc 0} + (1-t)\nu_1, \quad t \in [0,1],
}
where $\nu_1(\{f \in \s : \|f\| = 1\}) = 1$, and both $\delta_{\vc 0}$ and $\nu_1$---the latter of which must exist if $\nu \neq \delta_{\vc 0}$---are elements of $\k$.

Next notice that $R$ has a unique maximum at $\vc 0$.
Indeed, for any $f \neq \vc 0$ in $\s$, Jensen's inequality shows
\eq{
R(f) &=\EE \log \bigg(\sum_{u \in \N \times \Z^d} \sum_{v \sim u} f(v)e^{\beta \eta_u}P(v,u) + (1-\|f\|)e^{\lambda(\beta)}\bigg) \\
 &< \log \EE\bigg[\sum_{u \in \N \times \Z^d} \sum_{v \sim u} f(v)e^{\beta \eta_u}P(v,u) + (1-\|f\|)e^{\lambda(\beta)}\bigg] \\
 &=\log\bigg( \sum_{u \in \N \times\Z^d}\sum_{v\sim u} f(v)e^{\lambda(\beta)}P(v,u) + (1-\|f\|)e^{\lambda(\beta)}\bigg) = \lambda(\beta) = R(\vc 0).
}
In particular, when $0 < t < 1$ we have
\eq{
\r(\nu) = t\r(\delta_{\vc 0}) + (1-t)\r(\nu_1) = tR(\vc 0) + (1-t)\r(\nu_1) \in (\r(\nu_1),\lambda(\beta)).
}
Hence
\eq{
\beta > \beta_c \quad \stackrel{\mbox{\scriptsize\eqref{above_phase_transition}}}{\Rightarrow} \quad
\lambda(\beta) > p(\beta)
\stackrel{\mbox{\scriptsize\eqref{variational_formula}}}{=} \inf_{\nu \in \k} \r(\nu)
\quad \Rightarrow \quad
\nu \in \m \text{ only if } \nu = \nu_1.
}
That is, claim (b) holds.
On the other hand,
\eq{
0 \leq \beta \leq \beta_c \quad \stackrel{\mbox{\scriptsize\eqref{below_phase_transition}}}{\Rightarrow} \quad
\lambda(\beta)  = p(\beta)
\stackrel{\mbox{\scriptsize\eqref{variational_formula}}}{=} \inf_{\nu \in \k} \r(\nu)
\quad \Rightarrow \quad
\nu = \delta_{\vc 0} \text{ for all $\nu \in \k$},
}
and so claim (a) holds as well.
\end{proof}

\subsection{Asymptotic pure atomicity}
Recall that a sequence of random probability measures\linebreak $(\rho_i(\omega_i \in \cdot))_{i \geq 0}$ on $\Z^d$ is \textit{asymptotically purely atomic} if for every sequence $\eps_i \to 0$,
\eq{
\lim_{n\to\infty} \frac{1}{n} \sum_{i = 0}^{n-1} \rho_i(\omega_i \in \a_i^{\eps_i}) = 1 \quad \mathrm{a.s.},
}
where $\a_i^\eps \coloneqq \{x \in \Z^d : \rho_i(\omega_i = x) > \eps\}$ is the set of ``$\eps$-atoms" with respect to $\rho_i(\omega_i \in \cdot)$.
Let us recall the statement of Theorem \ref{apa_new}.

\begin{thm}[{cf.~\cite[Theorem 6.3]{bates-chatterjee17}}] \label{total_mass}
Assume \eqref{mgf_assumption}.
Then the following statements hold:
\begin{itemize}
\item[(a)] If $\beta > \beta_c$, then $(\rho_i(\omega_i \in \cdot))_{i \geq 0}$ is asymptotically purely atomic.
\item[(b)] If $0 \leq \beta \leq \beta_c$, then there is a sequence $(\eps_i)_{i \geq 0}$ tending to 0 as $i \to \infty$, such that
\eq{
\lim_{n \to \infty} \frac{1}{n} \sum_{i = 0}^{n-1} \rho_{i}(\omega_i \in \a_i^{\eps_i}) = 0 \quad \mathrm{a.s.}
}
\end{itemize}
\end{thm}

Since $(\s,d_\alpha)$ has an equivalent topology to $(\s,d_2)$ by Proposition \ref{metrics_equivalent}, continuous functionals on the latter space are also continuous on the former.
Consequently, the proof in \cite[Section 6]{bates-chatterjee17} given for the SRW case requires no modification.

\subsection{Geometric localization with positive density}
As usual, let $f_i(\cdot) = \rho_i(\omega_i = \cdot)$ denote the probability mass function for the $i$-th endpoint distribution.
We say that the sequence $(f_i)_{i \geq 0}$ exhibits \textit{geometric localization with positive density} if for every $\delta > 0$, there is $K < \infty$ and $\theta  > 0$ such that
\eq{
\liminf_{n \to \infty} \frac{1}{n} \sum_{i = 0}^{n-1} \one_{\{f_i \in \g_{\delta,K}\}} \geq \theta \quad \mathrm{a.s.},
}
where
\eq{
\g_{\delta,K} = \Big\{f : \Z^d \to [0,1] : \sum_{x \in \Z^d} f(x) = 1,\, \sum_{x \in D} f(x) > 1 - \delta \text{ for some $D\subset \Z^d$ with $\diam(D) \leq K$}\Big\},
}
and $\diam(D) \coloneqq \sup\{\|x-y\|_1 : x,y \in D\}$.
Here we restate Theorem \ref{geometric_new}.

\begin{thm}[{cf.~\cite[Theorem 7.3]{bates-chatterjee17}}] \label{localized_subsequence}
Assume \eqref{mgf_assumption}.
Then the following statements hold:
\begin{itemize}
\item[(a)] If $\beta > \beta_c$, then $(f_i)_{i \geq 0}$ is geometrically localized with positive density. 
\item[(b)] If $0 \leq \beta \leq \beta_c$, then for any $\delta \in (0,1)$ and any $K$,
\eq{
\lim_{n \to \infty} \frac{1}{n} \sum_{i = 0}^{n-1} \one_{\{f_i \in \g_{\delta,K}\}} = 0 \quad \mathrm{a.s.} 
}
\end{itemize}
\end{thm}

As for Theorem \ref{total_mass}, the proof is equivalent to the one in the SRW case, which the reader can find in \cite[Section 7]{bates-chatterjee17}.

\section{Acknowledgments}
I am very grateful to the referees for their comments, suggestions, and corrections.
This work also benefited from fruitful conservations with Amir Dembo, Persi Diaconis, and Ofer Zeitouni, as well as feedback from Sourav Chatterjee, with whom collaboration inspired this work.
I thank Subhabrata Sen and Lutfu Simsek for directing my attention to some of the referenced articles.


\bibliography{directed_polymers.bib}

\end{document}